\documentclass[final,leqno]{siamltex}
\usepackage[leqno]{amsmath}
\usepackage{graphicx}% Figure
\usepackage{graphics}
\usepackage{epstopdf}
\usepackage{subfigure}% Subfigure
\usepackage{threeparttable}

\usepackage{bm}
\usepackage{amssymb}
\usepackage{leftidx}
\usepackage{empheq}
\usepackage{float}
\usepackage{color}
\usepackage{makecell}
\usepackage{enumerate}
\allowdisplaybreaks
\renewcommand{\theequation}{\arabic{section}.\arabic{equation}}
\numberwithin{equation}{section}
\newtheorem{remark}{Remark}[section]
\newtheorem{example}{Example}[section]
\newtheorem{scheme}{Scheme}
\title{ An enhanced and highly efficient semi-implicit combined Lagrange multiplier approach with preserving original energy law
 for dissipative systems.
        \thanks{
We would like to acknowledge the assistance of volunteers in putting together this example manuscript and supplement. This work is supported by National Natural Science Foundation of China (Grant Nos: 12001336, 12271302, 12131014) and Natural Science Outstanding Youth Fund of Shandong Province (Grant No: ZR2023YQ007). N. Zheng is partially supported by the Hong Kong Polytechnic University Postodoctoral Research Fund 1-W22P.}}

      \author{Zhengguang Liu
             \thanks{School of Mathematics and Statistics, Shandong Normal University, Jinan, Shandong, 250358, China. Email: liuzhg@sdnu.edu.cn.}
             \and
         	 Nan Zheng$^1$ 
         	 \thanks{Department of Applied Mathematics, The Hong Kong Polytechnic University, Hung Hom, Hong Kong. Email: nanzheng@polyu.edu.hk.}
             \and
             Xiaoli Li$^1$
             \thanks{School of Mathematics, Shandong University, Jinan, Shandong, 250100, China. Email: xiaomath@sdu.edu.cn.}
             }     
\begin{document}
\footnotetext[1]{Corresponding authors.} 
%\UseRawInputEncoding
\maketitle

\begin{abstract}
Recently, a new Lagrange multiplier approach was introduced by Cheng, Liu and Shen in \cite{cheng2020new}, which has been broadly used to solve various challenging phase field problems. To design original energy stable schemes, they have to solve a nonlinear algebraic equation to determine the introduced Lagrange multiplier, which can be computationally expensive, especially for large-scale and long-time simulations involving complex nonlinear terms. This paper presents an essential improved technique to modify this issue, which can be seen as a semi-implicit combined Lagrange multiplier approach. In general, the new constructed schemes keep all the advantages of the Lagrange multiplier method and significantly reduce the computation costs. Besides, the new proposed BDF2 scheme dissipates the original energy, as opposed to a modified energy for the classical Lagrange multiplier approach in \cite{cheng2020new}. We further construct high-order BDF$k$ schemes based on the new proposed approach. In addition, we establish a general framework for extending our constructed method to dissipative systems.  Finally several examples have been presented to demonstrate the effectiveness of the proposed approach.
\end{abstract}



\pagestyle{myheadings}
\thispagestyle{plain}
\markboth{Zhengguang Liu, Nan Zheng, Xiaoli Li} {A semi-implicit combined Lagrange multiplier approach}
   %==================================================================
\section{Introduction}
Dissipative systems are indeed widespread in various practical problems. The design of efficient and accurate energy stable schemes for nonlinear dissipative systems, such as phase field models, has been a subject of extensive research in the past few decades. Phase field models are derived from mathematical descriptions of interface behavior in multi-phase materials. These models have found wide applications in interface problems such as fluid dynamics, environmental science, and mechanics of materials, etc \cite{chen2002phase,keller1970initiation,osher1988fronts,rudin1992nonlinear}. The typical applications include the synthesis of advanced composite material, the complex multi-phase fluid, integrated circuits, lithium-ion batteries, 3D printing, etc \cite{fallah2012phase,lowen2010phase,wang2020application,zuo2015phase}. Phase field models, including the Allen-Cahn equation \cite{feng2003numerical,ilmanen1993convergence,shen2010numerical}, Cahn-Hilliard equation \cite{cahn1958free,cahn1959free,eyre1998unconditionally},  molecular beam epitaxy models \cite{cho1975molecular,joyce1985molecular,yang2017numerical}, phase field crystal models \cite{elder2002modeling,wise2009energy,WuPhase}, are widely used in various applications.

We will first present our numerical schemes to simulate the phase field models and further extend it to general dissipative systems. In general, the phase field models are dynamically driven by a free energy $E(\phi)$, and take the following gradient flow form:
\begin{equation}\label{intro-e1}
\displaystyle\frac{\partial \phi}{\partial t}=-\mathcal{G}\frac{\delta E}{\delta\phi},
\end{equation}
with periodic or homogeneous Neumann boundary condition, and $E(\phi)=\frac{1}{2}(\mathcal{L}\phi,\phi)+\int_\Omega F(\phi(\textbf{x}))d\textbf{x}$, $\mathcal{G}$ and $\mathcal{L}$ are both positive definite operators.  

One can easily to find that the phase field models \eqref{intro-e1} satisfy a dissipative energy law:
\begin{equation*}
\displaystyle\frac{d}{dt}E=\left(\frac{\delta E}{\delta\phi},\frac{\partial \phi}{\partial t}\right)=-\left(\frac{\delta E}{\delta\phi},\mathcal{G}\frac{\delta E}{\delta\phi}\right)\leq0.
\end{equation*}

Due to the complex properties of the phase field models such as high-order nonlinearities and physical constraints, it is still quite a challenging interdisciplinary project to study the application of phase field models in efficient simulations of the interface evolution in complex physical process. Many scholars have tried many approaches to develop efficient, easy-to-implement, energy stable numerical schemes to accurately capture the dynamics of interface singularities as well as the micro-structures for the derived multi-phase complex material systems. The classical approaches are the fully explicit scheme and fully implicit scheme \cite{du1991numerical,feng2004error}. The full explicit scheme, while simple, has very strict time step limitations. For fully implicit scheme, the nonlinear problems have to be solved in each time step. Meanwhile, the existence and uniqueness of the solution usually have strong restrictions on the time step, which limits its wide applicability. For more efficient long-time numerical simulations, the semi-implicit scheme is a good choice.  The more widely used and effective methods mainly include convex splitting methods \cite{baskaran2013convergence,eyre1998unconditionally}, stabilized semi-implicit methods \cite{chen1998applications,shen2010numerical,xu2006stability}, exponential time-differencing (ETD) methods \cite{du2019maximum,du2021maximum,ju2018energy}, invariant energy quadratization (IEQ) methods \cite{yang2018linear,yang2017numerical,zhao2017numerical}, scalar auxiliary variable (SAV) methods \cite{huang2020highly,shen2018scalar,ShenA}, Lagrange multiplier methods \cite{cheng2020new,cheng2022new}, etc.

Recently, the SAV approach introduced by Shen et al. \cite{shen2018scalar,ShenA} has been attracted much attention in numerical solutions for phase field models due to its inherent advantage of preserving energy dissipation law. However, the unconditional energy stability is with respect to a modified energy according to the auxiliary variables instead of the original variables. To design unconditionally energy stable schemes with the original energy, Cheng, Liu and Shen \cite{cheng2020new} introduce a new Lagrange multiplier approach for gradient flows. Compared with the SAV approach, the new Lagrange multiplier enjoys two additional advantages: (i) the numerical schemes satisfy an original energy dissipation law; (ii) they do not require an assumption that the nonlinear part of the free energy to be bounded from below. However, the trade-off for these advantages is that a nonlinear algebraic equation for the Lagrange multiplier needs to be solved essentially, which significantly increases the computational costs, especially for phase field models requiring long-time simulations. The Newton iteration method is particularly inefficient for models with non-algebraic type nonlinear terms.

The main purpose of this paper is to construct a new semi-implicit combined Lagrange multiplier approach to improve its efficiency in computational costs. compared with the classical Lagrange multiplier method in \cite{cheng2020new}, the new proposed method enjoys the following advantages:

$\bullet$ It is unconditionally energy stable with respect to the original free energy, and provides accuracy and efficiency comparable to the classical Lagrange multiplier approach;

$\bullet$ It significantly reduces the computation costs compared to with the classical Lagrange multiplier approach;

$\bullet$ The BDF2 schemes based on the new approach dissipate the original energy, as opposed to a modified energy in \cite{cheng2020new};

$\bullet$ We construct several high-order BDF$k$ schemes based on the new proposed approach and establish a general framework for extending our constructed method to dissipative systems. 

The rest of this paper is organized as follows. In Section 2, we provide a brief review of the new Lagrange multiplier approach introduced by Cheng, Liu and Shen \cite{cheng2020new}. In Section 3, we present the second-order Crank-Nicloson and BDF2 schemes based on a new semi-implicit combined Lagrange multiplier approach. All discrete schemes are proved the unconditional energy stability. A high-order BDF$k$ scheme based on the improved technique is considered in Section 4. In Section 5, we further extend the proposed method to the general dissipative systems. In Section 6, we give some comparisons of the proposed approach with the new Lagrange multiplier approach to validate its high efficiency.

\section{A brief review of the new Lagrange Multiplier approach}
In this section, we give a brief review of the new Lagrange Multiplier approach for phase field models \eqref{intro-e1} to better introduce our newly proposed methods. 

Introduce a scalar auxiliary function $\eta(t)$ with $\eta(0)=1$, and reformulate the phase field models \eqref{intro-e1} as the following:
\begin{equation*}
   \begin{array}{l}
\displaystyle\frac{\partial \phi}{\partial t}=-\mathcal{G}\mu,\\
\mu=\mathcal{L}\phi+\eta(t)F'(\phi),\\
\displaystyle\frac{d}{dt}\int_\Omega F(\phi)d\textbf{x}=\eta(t)(F'(\phi),\phi_t).
   \end{array}
  \end{equation*}
Taking the inner products of the first two equations in the above with $\mu$ and $-\phi_t$ respectively, summing up the results together with the third equation, one can obtain the following original energy dissipation law:
\begin{equation*}
\displaystyle\frac{d}{dt}E=-(\mathcal{G}\mu,\mu).
\end{equation*}

Next, we give a second-order Crank–Nicolson scheme for above equivalent phase field models:
\begin{scheme}
\begin{equation}\label{nLM-e1}
   \begin{array}{l}
\displaystyle\frac{\phi^{n+1}-\phi^n}{\Delta t}=-\mathcal{G}\mu^{n+\frac12},\\
\displaystyle\mu^{n+\frac12}=\frac12\mathcal{L}\phi^{n+1}+\frac12\mathcal{L}\phi^{n}+\eta^{n+\frac12}F'(\phi^{*,n+\frac12}),\\
\displaystyle(F(\phi^{n+1}),1)-(F(\phi^{n}),1)=\eta^{n+\frac12}(F'(\phi^{*,n+\frac12}),\phi^{n+1}-\phi^n),
   \end{array}
  \end{equation}  
where $\phi^{*,n+\frac12}=\frac32\phi^n-\frac12\phi^{n-1}$ for $n\geq1$.
\end{scheme}

Taking the inner products of the first two equations in the above scheme \eqref{nLM-e1} with $\Delta t\mu^{n+\frac12}$ and $-(\phi^{n+1}-\phi^n)$,
summing up the results together with the third equation, we obtain immediately the following energy dissipation law:
\begin{equation*}
\displaystyle\left[\frac12(\mathcal{L}\phi^{n+1},\phi^{n+1})+(F(\phi^{n+1}),1)\right]-\left[\frac12(\mathcal{L}\phi^{n},\phi^{n})+(F(\phi^{n}),1)\right]=-\Delta t(\mu^{n+\frac12},\mathcal{G}\mu^{n+\frac12}).
\end{equation*}

We now show how to solve the scheme \eqref{nLM-e1} efficiently. Substituting the second equation into the first equation for scheme \eqref{nLM-e1}, we can obtain the following linear matrix equation
\begin{equation*}
\aligned
(I+\frac12\Delta t\mathcal{G}\mathcal{L})\phi^{n+1}=(I-\frac12\Delta t\mathcal{G}\mathcal{L})\phi^{n}-\eta^{n+\frac12}\Delta t\mathcal{G}F'(\phi^{*,n+\frac12}).
\endaligned
\end{equation*}
Setting $\phi^{n+1}=\phi_1^{n+1}+\eta^{n+\frac12}q^{n+1}$, one can find that $\phi_1^{n+1}$ and $q^{n+1}$ are solutions of the following two linear equations with constant coefficients:
\begin{equation}\label{nLM-e2}
\aligned
(I+\frac12\Delta t\mathcal{G}\mathcal{L})\phi_1^{n+1}=(I-\frac12\Delta t\mathcal{G}\mathcal{L})\phi^{n},\quad (I+\frac12\Delta t\mathcal{G}\mathcal{L})q^{n+1}=-\Delta t\mathcal{G}F'(\phi^{*,n+\frac12}).
\endaligned
\end{equation}
Once $\phi_1^{n+1}$ and $q^{n+1}$ are known, we can determine $\eta^{n+\frac12}$ by the following nonlinear algebraic equation:
\begin{equation}\label{nLM-e3}
\displaystyle\left(F(\phi_1^{n+1}+\eta^{n+\frac12}q^{n+1})-F(\phi^{n}),1\right)=
\eta^{n+\frac12}\left({F'}(\phi^{*,n+\frac12}),\phi_1^{n+1}+\eta^{n+\frac12}q^{n+1}-\phi^n\right).
\end{equation} 
\begin{remark}
The complexity of the nonlinear equation \eqref{nLM-e3} depends on $F(\phi)$. If $F(\phi)$ is a polynomial function such as double well potential, the equation \eqref{nLM-e3} will be an algebraic equation of $\eta^{n+\frac12}$. Using a Newton iteration with 1 as the initial condition will be very high efficiency. However, if $F(\phi)$ is not a polynomial function such as $F(\phi)=-\frac12\ln(1+|\nabla \phi|^2)$ in the molecular beam epitaxial model, the Newton iteration will become inefficient which maybe obtain incorrect iterative solution.  
\end{remark}
\section{A novel semi-implicit combined Lagrange multiplier approach}
From above introduction of new Lagrange Multiplier approach in Section 2, one can find that we need to solve a nonlinear algebraic equation about $\eta^{n+\frac12}$ in advance to obtain $\phi^{n+1}$ at each time step. This may spend a large computational costs in long time numerical simulations. In this section, we consider an improved technique to modify this issue, which can be seen as a semi-implicit combined Lagrange Multiplier appraoch. Before introducing the method to be proposed, we give a novel Lagrange multiplier approach which is essentially equivalent with the introduced new Lagrange multiplier approach, more detailed can be seen in \cite{liu2022novel}.
  
Introduce a zero-factor Lagrange Multiplier as:
\begin{equation}\label{LM-e2}
\eta(t)=\frac{d}{dt}\left[(F(\phi),1)-(F(\phi),1)\right]=\frac{d}{dt}\int_\Omega F(\phi)d\textbf{x}-\int_\Omega F'(\phi)\phi_td\textbf{x}=0.
\end{equation}
Then one can rewrite the phase field models \eqref{intro-e1} with $\eta(t)$ as follows:
\begin{equation*}
   \begin{array}{l}
\displaystyle\frac{\partial \phi}{\partial t}=-\mathcal{G}\mu,\\
\mu=\mathcal{L}\phi+F'(\phi)+\eta(t)F'(\phi).
   \end{array}
  \end{equation*}
In order to keep the energy dissipative law of the equivalent system and noting $\eta(t)=0$ at the continuous level, we then let $\eta(t)=\eta(t)\int_\Omega F'(\phi)\phi_td\textbf{x}$ to obtain the following equation:
\begin{equation*}
\aligned
\displaystyle\frac{d}{dt}\int_\Omega F(\phi)d\textbf{x}
&=\displaystyle\int_\Omega F'(\phi)\phi_td\textbf{x}+\eta(t)\\
&=\displaystyle\int_\Omega F'(\phi)\phi_td\textbf{x}+\eta(t)\int_\Omega F'(\phi)\phi_td\textbf{x}.
\endaligned
\end{equation*}
Therefore, we can easily obtain the following equivalent systems of the original phase field models \eqref{intro-e1}:
\begin{equation}\label{LM-e3}
   \begin{array}{l}
\displaystyle\frac{\partial \phi}{\partial t}=-\mathcal{G}\mu,\\
\mu=\mathcal{L}\phi+F'(\phi)+\eta(t)F'(\phi),\\
\displaystyle\frac{d}{dt}\int_\Omega F(\phi)d\textbf{x}=\displaystyle\int_\Omega F'(\phi)\phi_td\textbf{x}+\eta(t)\int_\Omega F'(\phi)\phi_td\textbf{x}.
   \end{array}
  \end{equation}
Taking the inner products of the first two equations in the above with $\mu$ and $-\phi_t$ respectively, summing up the results together with the third equation, one can also obtain the following original energy dissipation law:
\begin{equation*}
\displaystyle\frac{d}{dt}E=-(\mathcal{G}\mu,\mu).
\end{equation*}
\subsection{Second-order Crank-Nicolson scheme}\label{s3.1}
In this subsection, we will give a modified Lagrange multiplier approach for the equivalent system \eqref{LM-e3}. Similar as \eqref{nLM-e1}, a classic Lagrange multiplier second-order scheme based on Crank-Nicolson (CN) discretization can be constructed as follows:
\begin{equation}\label{im-LM-e1}
   \begin{array}{l}
\displaystyle\frac{\phi^{n+1}-\phi^n}{\Delta t}=-\mathcal{G}\mu^{n+\frac12},\\
\displaystyle\mu^{n+\frac12}=\frac12\mathcal{L}\phi^{n+1}+\frac12\mathcal{L}\phi^{n}+F'(\phi^{*,n+\frac12})+\eta^{n+1/2}F'({\phi}^{*,n+\frac12}),\\
\displaystyle(F(\phi^{n+1}),1)-(F(\phi^{n}),1)=\displaystyle\left(F'({\phi}^{*,n+\frac12}),\phi^{n+1}-\phi^n\right)+\eta^{n+1/2}\left(F'({\phi}^{*,n+\frac12}),\phi^{n+1}-\phi^n\right),
   \end{array}
  \end{equation}
where ${\phi}^{*,n+\frac12}=\frac32\phi^n-\frac12\phi^{n-1}$.

Taking the inner products of the first two equations in the above scheme \eqref{im-LM-e1} with $\Delta t\mu^{n+\frac12}$ and $-(\phi^{n+1}-\phi^n)$,
summing up the results together with the third equation, we can also obtain immediately the following original energy dissipation law:
\begin{equation*}
\displaystyle\left[\frac12(\mathcal{L}\phi^{n+1},\phi^{n+1})+(F(\phi^{n+1}),1)\right]-\left[\frac12(\mathcal{L}\phi^{n},\phi^{n})+(F(\phi^{n}),1)\right]=-\Delta t(\mu^{n+\frac12},\mathcal{G}\mu^{n+\frac12}).
\end{equation*}

The above second-order scheme \eqref{im-LM-e1} is nonlinear for the variables $\phi^{n+1}$ and $\eta^{n+1}$. We now show how to solve it efficiently. Combining the first two equations in \eqref{im-LM-e1}, we can obtain the following linear matrix equation
\begin{equation*}
\aligned
(I+\frac12\Delta t\mathcal{G}\mathcal{L})\phi^{n+1}=(I-\frac12\Delta t\mathcal{G}\mathcal{L})\phi^{n}-\Delta t\mathcal{G}F'({\phi}^{*,n+\frac12})-\eta^{n+\frac12}\Delta t\mathcal{G}F'({\phi}^{*,n+\frac12}).
\endaligned
\end{equation*}

Noting that the coefficient matrix $A=(I+\frac12\Delta t\mathcal{G}\mathcal{L})$ is a symmetric positive matrix, then we obtain
\begin{equation}\label{im-LM-e2}
\aligned
\phi^{n+1}
&=A^{-1}\left[(I-\frac12\Delta t\mathcal{G}\mathcal{L})\phi^{n}-\Delta t\mathcal{G}F'({\phi}^{*,n+\frac12})\right]-\eta^{n+\frac12}\Delta tA^{-1}\mathcal{G}F'({\phi}^{*,n+\frac12})\\
&=\overline{\phi}^{n+1}+\eta^{n+\frac12}q^{n+1},
\endaligned
\end{equation}
Here $\overline{\phi}^{n+1}$ and $q^{n+1}$ can be solved directly by $\phi^n$ and ${\phi}^{*,n+\frac12}$ as follows:
\begin{equation}\label{im-LM-e3}
\aligned
\overline{\phi}^{n+1}=A^{-1}\left[(I-\frac12\Delta t\mathcal{G}\mathcal{L})\phi^{n}-\Delta t\mathcal{G}F'({\phi}^{*,n+\frac12})\right],\quad q^{n+1}=-\Delta tA^{-1}\mathcal{G}F'({\phi}^{*,n+\frac12}).
\endaligned
\end{equation}
\begin{remark}\label{remark1}
The intermediate variable $\overline{\phi}^{n+1}$ is the solution of the classic semi-implicit scheme. Thus from equation $\phi^{n+1}=\overline{\phi}^{n+1}+\eta^{n+\frac12}q^{n+1}$, one can know that the solution of Lagrange multiplier scheme \eqref{im-LM-e1} is actually a solution modification of the classic semi-implicit scheme. 
\end{remark}

To reduce the computational costs while maintaining unconditionally original energy dissipation law,  we can modify the classic Lagrange multiplier scheme \eqref{im-LM-e1} as follows:

Given $\phi^n$, $\phi^{n-1}$, we compute $\phi^{n+1}$, $\eta^{n+\frac12}$ via the following three steps:
\begin{scheme}

\textbf{Step I}: Calculate the intermediate solution $\overline{\phi}^{n+1}$ from the following classic semi-implicit CN scheme:
\begin{equation}\label{im-LM-e4}
   \begin{array}{l}
\displaystyle\frac{\overline{\phi}^{n+1}-\phi^n}{\Delta t}=-\mathcal{G}\mu^{n+\frac12},\\
\displaystyle\mu^{n+\frac12}=\frac12\mathcal{L}\overline{\phi}^{n+1}+\frac12\mathcal{L}\phi^{n}+F'({\phi}^{*,n+\frac12})
   \end{array}
  \end{equation}
where ${\phi}^{*,n+\frac12}=\frac32\phi^n-\frac12\phi^{n-1}$.

\textbf{Step II}: compute the Lagrange multiplier variable $\eta^{n+1/2}$ as follows:
 \begin{equation}\label{im-LM-e5}
\eta^{n+\frac12}=\left\{
   \begin{array}{ll}
0,\quad E(\overline{\phi}^{n+1})\leq E(\phi^n),\\
\overline{\eta}^{n+\frac12},\quad E(\overline{\phi}^{n+1})>E(\phi^n).
   \end{array}
   \right.
\end{equation} 
where $E(\phi)=\frac12(\mathcal{L}\phi,\phi)+(F(\phi),1)$ and $\overline{\eta}^{n+\frac12}$ can be obtained as follows:
\begin{equation}\label{im-LM-e6}
\aligned
&\left(F(\overline{\phi}^{n+1}+\overline{\eta}^{n+\frac12}q^{n+1})-F(\phi^{n}),1\right)=
\left[1+\overline{\eta}^{n+\frac12}\right]\left({F'}({\phi}^{*,n+\frac12}),\overline{\phi}^{n+1}+\overline{\eta}^{n+\frac12}q^{n+1}-\phi^n\right),
\endaligned
\end{equation}
where $q^{n+1}$ can be solved directly by $\phi^n$ and ${\phi}^{*,n+\frac12}$ as follows:
\begin{equation}\label{im-LM-e7}
\aligned 
q^{n+1}=-\frac12\Delta t\mathcal{G}\mathcal{L}q^{n+1}-\Delta t\mathcal{G}F'({\phi}^{*,n+\frac12}).
\endaligned
\end{equation}

\textbf{Step III}: Update $\phi^{n+1}$ as
\begin{equation}\label{im-LM-e8}
\phi^{n+1}=\overline{\phi}^{n+1}+\eta^{n+\frac12}q^{n+1}.
\end{equation}  
\end{scheme}

Noting that the value of $\eta^{n+\frac12}$ has no influence of the energy dissipation law, then the scheme \eqref{im-LM-e4}-\eqref{im-LM-e8} still holds the following original energy stability:
\begin{theorem}\label{im-LM-theorem1}
The second-order Crank-Nicloson semi-implicit combined Lagrange multiplier scheme \eqref{im-LM-e4}-\eqref{im-LM-e8} is unconditionally energy stable in the sense that
\begin{equation*}
\displaystyle E^{n+1}-E^n=\left[\frac12(\mathcal{L}\phi^{n+1},\phi^{n+1})+(F(\phi^{n+1}),1)\right]-\left[\frac12(\mathcal{L}\phi^{n},\phi^{n})+(F(\phi^{n}),1)\right]\leq0.
\end{equation*}
\end{theorem}
\begin{proof}
From equation \eqref{im-LM-e8}, one can see that $\phi^{n+1}=\overline{\phi}^{n+1}+\eta^{n+\frac12}q^{n+1}$. If $\eta^{n+\frac12}=0$, we have $\phi^{n+1}=\overline{\phi}^{n+1}$ and $E(\overline{\phi}^{n+1})\leq E(\phi^n)$. It implies that 
\begin{equation*}
\displaystyle E(\phi^{n+1})=E(\overline{\phi}^{n+1})\leq E(\phi^n).
\end{equation*}

If $\eta^{n+\frac12}\neq0$, we have $\eta^{n+\frac12}=\overline{\eta}^{n+\frac12}$ and $\phi^{n+1}=\overline{\phi}^{n+1}+\overline{\eta}^{n+\frac12}q^{n+1}$. Then the following equation will hold
\begin{equation}\label{im-LM-e9}
\aligned
&\left(F(\phi^{n+1})-F(\phi^{n}),1\right)=
\left[1+\overline{\eta}^{n+\frac12}\right]\left({F'}({\phi}^{*,n+\frac12}),\phi^{n+1}-\phi^n\right),
\endaligned
\end{equation}  

Bringing equation $\overline{\phi}^{n+1}=\phi^{n+1}-\overline{\eta}^{n+\frac12}q^{n+1}$ into equation \eqref{im-LM-e4}, we obtain
\begin{equation}\label{im-LM-e10}
   \begin{array}{l}
\displaystyle\phi^{n+1}-\phi^n=-\Delta t\mathcal{G}\mu^{n+\frac12}+\overline{\eta}^{n+\frac12}q^{n+1},\\
\displaystyle\mu^{n+\frac12}=\frac12\mathcal{L}\phi^{n+1}+\frac12\mathcal{L}\phi^{n}+F'({\phi}^{*,n+\frac12})-\overline{\eta}^{n+\frac12}\frac12\mathcal{L}q^{n+1}
   \end{array}
  \end{equation}

Define $\mu_1^{n+\frac12}=\frac12\mathcal{L}q^{n+1}+F'({\phi}^{*,n+\frac12})$ and combine it with equation \eqref{im-LM-e7}, we have 
\begin{equation}\label{im-LM-e11}
\aligned
q^{n+1}=-\Delta t\mathcal{G}\mu_1^{n+\frac12}
\endaligned
\end{equation}  

Substituting the equation \eqref{im-LM-e11} into \eqref{im-LM-e10}, we immediately obtain
\begin{equation}\label{im-LM-e12}
   \begin{array}{l}
\displaystyle\phi^{n+1}-\phi^n=-\Delta t\mathcal{G}(\mu^{n+\frac12}+\overline{\eta}^{n+\frac12}\mu_1^{n+\frac12}),\\
\displaystyle\mu^{n+\frac12}+\overline{\eta}^{n+\frac12}\mu_1^{n+\frac12}=\frac12\mathcal{L}\phi^{n+1}+\frac12\mathcal{L}\phi^{n}+F'({\phi}^{*,n+\frac12})+\overline{\eta}^{n+\frac12}F'({\phi}^{*,n+\frac12}).
   \end{array}
  \end{equation}

Taking the inner products of the equations in the above \eqref{im-LM-e12} with $\mu^{n+\frac12}+\overline{\eta}^{n+\frac12}\mu_1^{n+\frac12}$ and $-(\phi^{n+1}-\phi^n)$,
summing up the results together with the equation \eqref{im-LM-e9}, we can also obtain immediately the following original energy dissipation law:
\begin{equation*}
\aligned
&\displaystyle\left[\frac12(\mathcal{L}\phi^{n+1},\phi^{n+1})+(F(\phi^{n+1}),1)\right]-\left[\frac12(\mathcal{L}\phi^{n},\phi^{n})+(F(\phi^{n}),1)\right]\\
&=-\Delta t\left(\mu^{n+\frac12}+\overline{\eta}^{n+\frac12}\mu_1^{n+\frac12},\mathcal{G}(\mu^{n+\frac12}+\overline{\eta}^{n+\frac12}\mu_1^{n+\frac12})\right),
\endaligned
\end{equation*}
which completes the proof.
\end{proof}
\begin{remark}\label{im-LM-remark1} 
Compared with the classic Lagrange multiplier schemes \eqref{nLM-e1} and \eqref{im-LM-e1}, one can find that the new proposed schemes \eqref{im-LM-e4}-\eqref{im-LM-e7} only need to solve a nonlinear algebraic equation about $\eta^{n+\frac12}$ in very few time layers. Meanwhile, the new algorithm only requires solving one linear system with constant coefficients as opposed to the two linear systems by the classic Lagrange multiplier approach which will save much computational costs.
\end{remark}
\subsection{Second-order BDF2 scheme}
In this subsection, we will give a semi-implicit combined Lagrange multiplier scheme based on second-order backward difference formula (BDF2). From \cite{cheng2020new}, one can see that the classic Lagrange multiplier BDF2 scheme does not hold the original energy dissipation law. We next give a modified Lagrange multiplier BDF2 scheme for the equivalent system \eqref{LM-e3} which the original energy is dissipative. 

Firstly, the new Lagrange multiplier BDF2 scheme for the equivalent system \eqref{LM-e3} can be written as follows: 
\begin{equation}\label{im-LM-bdf2-e1}
   \begin{array}{l}
\displaystyle\frac{3\phi^{n+1}-4\phi^n+\phi^{n-1}}{2\Delta t}=-\mathcal{G}\mu^{n+1},\\
\displaystyle\mu^{n+1}=\mathcal{L}\phi^{n+1}+F'(\widehat{\phi}^{n+1})+\eta^{n+1}F'(\widehat{\phi}^{n+1}),\\
\displaystyle(3F(\phi^{n+1}),1)-(4F(\phi^{n})-F(\phi^{n-1}),1)=\displaystyle(1+\eta^{n+1})\left(F'(\widehat{\phi}^{n+1}),3\phi^{n+1}-4\phi^n+\phi^{n-1}\right),
   \end{array}
  \end{equation}
where $\widehat{\phi}^{n+1}=2\phi^n-\phi^{n-1}$.

Similarly as before, we can also set $\phi^{n+1}=\overline{\phi}^{n+1}+\eta^{n+1}q^{n+1}$. Here $\overline{\phi}^{n+1}$ is the solution of the classic semi-implicit BDF2 scheme and its value together with $q^{n+1}$ can be solved directly by $\phi^n$, $\phi^{n-1}$ and $\widehat{\phi}^{n+1}$ as follows:
\begin{equation}\label{im-LM-bdf2-e2}
\aligned
\overline{\phi}^{n+1}=A^{-1}\left[(2\phi^{n}-\frac12\phi^{n-1})-\Delta t\mathcal{G}F'(\widehat{\phi}^{n+1})\right],\quad q^{n+1}=-\Delta tA^{-1}\mathcal{G}F'(\widehat{\phi}^{n+1}),
\endaligned
\end{equation}   
where $A=(\frac32I+\Delta t\mathcal{G}\mathcal{L})$.

To keep the original dissipation law, we need 
\begin{equation}\label{im-LM-bdf2-e3}
\displaystyle E^{n+1}=\left[\frac12(\mathcal{L}\phi^{n+1},\phi^{n+1})+(F(\phi^{n+1}),1)\right]\leq\left[\frac12(\mathcal{L}\phi^{n},\phi^{n})+(F(\phi^{n}),1)\right]=E^n.
\end{equation}

Noting that $\phi^{n+1}=\overline{\phi}^{n+1}+\eta^{n+1}q^{n+1}$ and combining it with the third equation in \eqref{im-LM-bdf2-e1}, we obtain
\begin{equation*}
\aligned
\displaystyle E^{n+1}
&=\frac12(\mathcal{L}\phi^{n+1},\phi^{n+1})+(F(\phi^{n+1}),1)\\
&=\frac12(\mathcal{L}q^{n+1},q^{n+1})(\eta^{n+1})^2+(\mathcal{L}\overline{\phi}^{n+1},q^{n+1})\eta^{n+1}+\frac12(\mathcal{L}\overline{\phi}^{n+1},\overline{\phi}^{n+1})\\
&\quad+\left(F'(\widehat{\phi}^{n+1}),q^{n+1}\right)(\eta^{n+1})^2+\frac13\left(F'(\widehat{\phi}^{n+1}),3\overline{\phi}^{n+1}+3q^{n+1}-4\phi^n+\phi^{n-1}\right)\eta^{n+1}\\
&\quad+\frac13(4F(\phi^{n})-F(\phi^{n-1}),1)+\frac13\left(F'(\widehat{\phi}^{n+1}),3\overline{\phi}^{n+1}-4\phi^n+\phi^{n-1}\right)\\
&=a(\eta^{n+1})^2+b\eta^{n+1}+c,
\endaligned
\end{equation*}
where the coefficients $a$, $b$ and $c$ satisfy:
\begin{equation*}
\aligned
&a=\frac12(\mathcal{L}q^{n+1},q^{n+1})+\left({F'}(\widehat{\phi}^{n+1}),q^{n+1}\right),\\
&b=(\mathcal{L}\overline{\phi}^{n+1},q^{n+1})+\frac13\left({F'}(\widehat{\phi}^{n+1}),3\overline{\phi}^{n+1}+3q^{n+1}-4\phi^n+\phi^{n-1}\right),\\
&c=\frac12(\mathcal{L}\overline{\phi}^{n+1},\overline{\phi}^{n+1})+\frac13(4F(\phi^{n})-F(\phi^{n-1}),1)+\frac13\left({F'}(\widehat{\phi}^{n+1}),3\overline{\phi}^{n+1}-4\phi^n+\phi^{n-1}\right).
\endaligned
\end{equation*}

Then, the original dissipation law \eqref{im-LM-bdf2-e3}  will be transformed into 
\begin{equation*}
\aligned
\displaystyle E^{n+1}=a(\eta^{n+1})^2+b\eta^{n+1}+c\leq E^n.
\endaligned
\end{equation*}
\begin{remark}\label{im-LM-bdf2-lemma1}
(Optimal choice for $\eta^{n+1}$). Here we explain the optimal choice for $\eta^{n+1}$. $\eta^{n+1}$ can be chosen as a solution of the following optimization problem,
\begin{equation}\label{im-LM-bdf2-e4}
\aligned
\eta^{n+1}=\min|Q(\eta)|~s.t.~ a\eta^2+b\eta+c\leq E^n,
\endaligned
\end{equation}
where 
\begin{equation*}
\aligned
\displaystyle Q(\eta)=
&(3F(\overline{\phi}^{n+1}+\eta q^{n+1}),1)-(4F(\phi^{n})-F(\phi^{n-1}),1)\\
&-(1+\eta)\left(F'(\widehat{\phi}^{n+1}),3\overline{\phi}^{n+1}+3\eta q^{n+1}-4\phi^n+\phi^{n-1}\right).
\endaligned
\end{equation*}
\end{remark}

Similarly as the schemes \eqref{im-LM-e4}-\eqref{im-LM-e7}, we can compute $\phi^{n+1}$ by the following modified technique to save the computational costs.

Given $\phi^n$, $\phi^{n-1}$, we compute $\phi^{n+1}$, $\eta^{n+1}$ via the following three steps:

\begin{scheme}

\textbf{Step I}: Calculate the intermediate solution $\overline{\phi}^{n+1}$ from the following classic semi-implicit BDF2 scheme:
\begin{equation}\label{im-LM-bdf2-e5}
   \begin{array}{l}
\displaystyle\frac{3\overline{\phi}^{n+1}-4\phi^n+\phi^{n-1}}{2\Delta t}=-\mathcal{G}\mu^{n+1},\\
\displaystyle\mu^{n+1}=\mathcal{L}\overline{\phi}^{n+1}+F'(\widehat{\phi}^{n+1})
   \end{array}
  \end{equation}
where $\widehat{\phi}^{n+1}=2\phi^n-\phi^{n-1}$.

\textbf{Step II}: compute the Lagrange multiplier variable $\eta^{n+1}$ as follows:
 \begin{equation}\label{im-LM-bdf2-e6}
\eta^{n+1}=\left\{
   \begin{array}{ll}
0,\quad E(\overline{\phi}^{n+1})\leq E(\phi^n),\\
\overline{\eta}^{n+1},\quad E(\overline{\phi}^{n+1})>E(\phi^n).
   \end{array}
   \right.
\end{equation} 
where $\overline{\eta}^{n+1}$ can be obtained as follows:
\begin{equation*}
\aligned
\overline{\eta}^{n+1}=\min|Q(\eta)|~s.t.~ a\eta^2+b\eta+c\leq E^n,
\endaligned
\end{equation*}
where 
\begin{equation*}
\aligned
\displaystyle Q(\eta)=
&(3F(\overline{\phi}^{n+1}+\eta q^{n+1}),1)-(4F(\phi^{n})-F(\phi^{n-1}),1)\\
&-(1+\eta)\left(F'(\widehat{\phi}^{n+1}),3\overline{\phi}^{n+1}+3\eta q^{n+1}-4\phi^n+\phi^{n-1}\right).
\endaligned
\end{equation*}

\textbf{Step III}: Update $\phi^{n+1}$ as
\begin{equation}\label{im-LM-bdf2-e7}
\phi^{n+1}=\overline{\phi}^{n+1}+\eta^{n+1}q^{n+1}.
\end{equation}  
\end{scheme}

One can easy to obtain that the schemes \eqref{im-LM-bdf2-e5}-\eqref{im-LM-bdf2-e7} holds the following original energy stability:
\begin{theorem}\label{im-LM-bdf2-theorem1}
The second-order BDF2 scheme based on the semi-implicit combined Lagrange multiplier approach \eqref{im-LM-bdf2-e5}-\eqref{im-LM-bdf2-e7} is unconditionally energy stable in the sense that
\begin{equation*}
E^{n+1}=\left[\frac12(\mathcal{L}\phi^{n+1},\phi^{n+1})+(F(\phi^{n+1}),1)\right]\leq\left[\frac12(\mathcal{L}\phi^{n},\phi^{n})+(F(\phi^{n}),1)\right]=E^n.
\end{equation*}
\end{theorem}
\begin{proof}
From \textbf{Step III} \eqref{im-LM-bdf2-e7}, one knows that $\phi^{n+1}=\overline{\phi}^{n+1}+\eta^{n+1}q^{n+1}$. Noting that the choice of $\eta^{n+1}$ in \eqref{im-LM-bdf2-e6}, we obtain that if $\eta^{n+1}=0$, we have $\phi^{n+1}=\overline{\phi}^{n+1}$ and $E(\overline{\phi}^{n+1})\leq E(\phi^n)$. It implies that 
\begin{equation*}
\displaystyle E^{n+1}=E(\phi^{n+1})=E(\overline{\phi}^{n+1})\leq E(\phi^n)=E^n.
\end{equation*}

If $\eta^{n+1}\neq0$, we have $\eta^{n+1}=\overline{\eta}^{n+1}$ and $\phi^{n+1}=\overline{\phi}^{n+1}+\overline{\eta}^{n+1}q^{n+1}$. $\overline{\eta}^{n+1}$ can be obtained from Remark \ref{im-LM-bdf2-lemma1}. It means that $E^{n+1}=a(\overline{\eta}^{n+1})^2+b\overline{\eta}^{n+1}+c\leq E^n$. 
which completes the proof.
\end{proof}
\begin{remark}\label{im-LM-bdf2-lemma2}
The new Lagrange multiplier approach based on BDF2 in \cite{cheng2020new} only preserves modified energy rather than the original energy. However, our proposed scheme based on BDF2 can maintain the original energy law. 
\end{remark}
\section{The high-order BDF$k$ scheme}
In this section, we will consider a high-order BDF$k$ scheme based on the proposed semi-implicit combined Lagrange multiplier method. Inspired by the generalized scalar auxiliary variable method \cite{huang2020highly}, we can obtain the following BDF$k$ Lagrange multiplier scheme:
\begin{equation}\label{im-LM-bdfk-e1}
   \begin{array}{l}
\displaystyle\frac{\alpha_k\overline{\phi}^{n+1}-\mathcal{A}_k(\phi^n)}{\Delta t}=-\mathcal{G}\mu^{n+1},\\
\displaystyle\mu^{n+1}=\mathcal{L}\overline{\phi}^{n+1}+F'(\widehat{\phi}^{n+1}),\\
\phi^{n+1}=\left[1-(1-\eta^{n+1})^{k+1}\right]\overline{\phi}^{n+1},\\
\displaystyle \frac{E(\phi^{n+1})-E(\phi^{n})}{\Delta t}=-(\eta^{n+1})^2\left(\mathcal{G}\mu^{n+1},\mu^{n+1}\right),
   \end{array}
  \end{equation}
where $\widehat{\phi}^{n+1}$, $\alpha_k$, $\mathcal{A}_k$ can be derived by Taylor expansion. For the readers’ convenience, we provide them for $k=1,2,3,4$ below:

BDF1:
\begin{equation}\label{im-LM-bdfk-e2}
\alpha_1=1,\quad\mathcal{A}_1(\phi^{n})=\phi^n,\quad\widehat{\phi}^{n+1}=\phi^n.
\end{equation}
BDF2:
\begin{equation}\label{im-LM-bdfk-e3}
\alpha_2=\frac{3}{2},\quad\mathcal{A}_2(\phi^{n})=2\phi^n-\frac{1}{2}\phi^{n-1},\quad\widehat{\phi}^{n+1}=2\phi^n-\phi^{n-1}.
\end{equation}
BDF3:
\begin{equation}\label{im-LM-bdfk-e4}
\alpha_3=\frac{11}{6},\quad\mathcal{A}_3(\phi^{n})=3\phi^n-\frac{3}{2}\phi^{n-1}+\frac13\phi^{n-2},\quad\widehat{\phi}^{n+1}=3\phi^n-3\phi^{n-1}+\phi^{n-2}.
\end{equation}
BDF4:
\begin{equation}\label{im-LM-bdfk-e5}
\aligned
&\alpha_4=\frac{25}{12},\quad\mathcal{A}_4(\phi^{n})=4\phi^n-3\phi^{n-1}+\frac43\phi^{n-2}-\frac14\phi^{n-3},\quad\widehat{\phi}^{n+1}=4\phi^n-6\phi^{n-1}+4\phi^{n-2}-\phi^{n-3}.
\endaligned
\end{equation}

One can easy to obtain the original dissipation law from the last equation in \eqref{im-LM-bdfk-e1}. However, using Newton iteration maybe be ineffective for some complex energy $E$.

To improve the efficiency of the above scheme \eqref{im-LM-bdfk-e1}, we can use similar technique to modify it as follows

\begin{scheme}

\textbf{Step I}: Calculate the intermediate solution $\overline{\phi}^{n+1}$ from the following classic BDF$k$ scheme:
\begin{equation}\label{im-LM-bdfk-e6}
   \begin{array}{l}
\displaystyle\frac{\alpha_k\overline{\phi}^{n+1}-\mathcal{A}_k(\phi^n)}{\Delta t}=-\mathcal{G}\mu^{n+1},\\
\displaystyle\mu^{n+1}=\mathcal{L}\overline{\phi}^{n+1}+F'(\widehat{\phi}^{n+1}).
   \end{array}
  \end{equation}

\textbf{Step II}: compute the Lagrange multiplier variable $\eta^{n+1}$ as follows:
 \begin{equation}\label{im-LM-bdfk-e7}
\eta^{n+1}=\left\{
   \begin{array}{ll}
1,\quad E(\overline{\phi}^{n+1})\leq E(\phi^n),\\
\overline{\eta}^{n+1},\quad E(\overline{\phi}^{n+1})>E(\phi^n).
   \end{array}
   \right.
\end{equation} 
where $\overline{\eta}^{n+1}$ can be obtained as follows: 
\begin{equation*}
\aligned
E(\left[1-(1-\overline{\eta}^{n+1})^{k+1}\right]\overline{\phi}^{n+1})-E(\phi^{n})=-(\overline{\eta}^{n+1})^2\left(\mathcal{G}\mu^{n+1},\mu^{n+1}\right).
\endaligned
\end{equation*}

\textbf{Step III}: Update $\phi^{n+1}$ as
\begin{equation}\label{im-LM-bdfk-e8}
\phi^{n+1}=\left[1-(1-\eta^{n+1})^{k+1}\right]\overline{\phi}^{n+1}.
\end{equation}  
\end{scheme}

\begin{theorem}\label{im-LM-bdfk-theorem1}
The high-order BDF$k$ scheme based on the semi-implicit combined Lagrange multiplier approach \eqref{im-LM-bdfk-e6}-\eqref{im-LM-bdfk-e8} is unconditionally energy stable in the sense that
\begin{equation*}
E(\phi^{n+1})\leq E(\phi^n).
\end{equation*}
\end{theorem}
\begin{proof}
Firstly, from the choice of $\eta^{n+1}$ in \textbf{Step II} of above BDF$k$ scheme, we have $\eta^{n+1}=1$ under the condition of $E(\overline{\phi}^{n+1})\leq E(\phi^n)$. Combining it with \textbf{Step III}, we have 
\begin{equation}\label{im-LM-bdfk-e9}
\phi^{n+1}=\left[1-(1-\eta^{n+1})^{k+1}\right]\overline{\phi}^{n+1}=\overline{\phi}^{n+1},
\end{equation}  
which means that 
\begin{equation*}
E(\phi^{n+1})=E(\overline{\phi}^{n+1})\leq E(\phi^n).
\end{equation*}

Secondly, if $E(\overline{\phi}^{n+1})>E(\phi^n)$, we have $\eta^{n+1}=\overline{\eta}^{n+1}$ and 
\begin{equation*}
\aligned
E(\phi^{n+1})-E(\phi^{n})=-(\eta^{n+1})^2\left(\mathcal{G}\mu^{n+1},\mu^{n+1}\right)\leq0.
\endaligned
\end{equation*}
\end{proof}

\section{Extension to the general dissipative systems}
In this section, we try the proposed semi-implicit combined Lagrange multiplier approach to solve the general dissipative system to construct numerical schemes with original dissipation law. Consider a domain $\Omega$ in two or three dimensions and a dissipative system on this domain, whose dynamics is described by \cite{zhang2022generalized}:
\begin{equation}\label{nds-e1}
\aligned
\displaystyle\frac{\partial \textbf{u}}{\partial t}+\mathcal{A}\textbf{u}+g(\textbf{u})=0,
\endaligned
\end{equation}
where $\textbf{u}(x,t)$ denotes the state variables of the system, $\mathcal{A}$ is a positive definite operator and $g(\textbf{u})$ is a semi-linear or quasi-linear operator. The above system satisfies the following energy dissipative law
\begin{equation}\label{nds-e2}
\aligned
\displaystyle\frac{dE(\textbf{u})}{dt}=-\mathcal{K}\textbf{u},
\endaligned
\end{equation}
where $\mathcal{K}\textbf{u}\geq0$ for all $\textbf{u}$.

It is natural and efficient to use the proposed technique to solve this dissipative system. Introduce a zero-factor Lagrange multiplier $\eta(t)$ and rewrite the dissipative system \eqref{nds-e1} with $\eta(t)$ as follows
\begin{equation}\label{nds-e3}
\aligned
&\displaystyle\frac{\partial \textbf{u}}{\partial t}+\mathcal{A}\textbf{u}+g(\textbf{u})+\eta(t)g(\textbf{u})=0,\\
&\displaystyle\frac{dE(\textbf{u})}{dt}=-\left[1+\eta(t)\right]\mathcal{K}\textbf{u}.
\endaligned
\end{equation}

Using similar technique as before, we can compute $\textbf{u}$ and $\eta$ as the following three steps:

\begin{scheme}

\textbf{Step I}: Calculate the intermediate solution $\overline{\textbf{u}}^{n+1}$ from the following classic semi-implicit CN scheme:
\begin{equation}\label{nds-e4}
   \begin{array}{l}
\displaystyle\frac{\overline{\textbf{u}}^{n+1}-\textbf{u}^{n}}{\Delta t}+\frac12\mathcal{A}(\overline{\textbf{u}}^{n+1}+\textbf{u}^{n})+g({\textbf{u}}^{*,n+\frac12})=0,
   \end{array}
  \end{equation}
where ${\textbf{u}}^{*,n+\frac12}=\frac32\textbf{u}^n-\frac12\textbf{u}^{n-1}$.

\textbf{Step II}: Compute the Lagrange multiplier variable $\eta^{n+1/2}$ as follows:
 \begin{equation}\label{nds-e5}
\eta^{n+\frac12}=\left\{
   \begin{array}{ll}
0,\quad E(\overline{\textbf{u}}^{n+1})\leq E(\textbf{u}^n),\\
\overline{\eta}^{n+\frac12},\quad E(\overline{\textbf{u}}^{n+1})>E(\textbf{u}^n).
   \end{array}
   \right.
\end{equation} 
where $\overline{\eta}^{n+\frac12}$ can be obtained as follows:
\begin{equation}\label{nds-e6}
\aligned
&E(\overline{\textbf{u}}^{n+1}+\overline{\eta}^{n+\frac12}\textbf{u}_2^{n+1})-E(\textbf{u}^{n})=
\left[1+\overline{\eta}^{n+\frac12}\right]\mathcal{K}(\overline{\textbf{u}}^{n+1}+\overline{\eta}^{n+\frac12}\textbf{u}_2^{n+1}),
\endaligned
\end{equation}
where $\textbf{u}_2^{n+1}$ can be solved directly by $\phi^n$ and ${\phi}^{*,n+\frac12}$ as follows:
\begin{equation}\label{nds-e7}
\aligned 
\textbf{u}_2^{n+1}+\frac12\Delta t\mathcal{A}\textbf{u}_2^{n+1}+\Delta tg({\textbf{u}}^{*,n+\frac12})=0.
\endaligned
\end{equation}

\textbf{Step III}: Update $\textbf{u}^{n+1}$ as
\begin{equation}\label{nds-e8}
\textbf{u}^{n+1}=\overline{\textbf{u}}^{n+1}+\eta^{n+\frac12}\textbf{u}_2^{n+1}.
\end{equation}  
\end{scheme}

Similarly as before, one can easy to obtain that the second-order scheme \eqref{nds-e5}-\eqref{nds-e8} holds the following dissipation law:
\begin{equation*}
E(\textbf{u}^{n+1})\leq E(\textbf{u}^n).
\end{equation*}

We next take the classic Navier-Stokes equation for example. Consider the following incompressible Navier-Stokes equations in $\Omega\times J$:
\begin{equation}\label{ns-e1}
\aligned
&\displaystyle\frac{\partial\textbf{u}}{\partial t}+\textbf{u}\cdot\nabla\textbf{u}-\nu\Delta\textbf{u}+\nabla p=0,\\
&\nabla\cdot\textbf{u}=0,
\endaligned
\end{equation}
where the domain $\Omega$ is in two or three dimensions with a sufficiently smooth boundary and $J=(0, T]$, $\textbf{u}$ and $p$ are the normalized velocity and pressure, $\nu>0$ denotes the inverse of the Reynolds number. We consider the periodic or homogeneous Dirichlet boundary conditions. The system \eqref{ns-e1} satisfies the following law
\begin{equation*}
	\frac{\mathrm{d}}{\mathrm{d} t} E(\mathbf{u})=-\nu\|\nabla\mathbf{u}\|^{2},
\end{equation*}
where $E(\mathbf{u})=\frac{1}{2}\|\mathbf{u}\|^{2}$ is the total energy. 

Introduce a Lagrange multiplier $\eta(t)$ and rewrite the Navier-Stokes equation \eqref{ns-e1} as the following equivalent system:
\begin{equation}\label{ns-e2}
\aligned
&\frac{\partial\textbf{u}}{\partial t}+\left[1+\eta(t)\right]\textbf{u}\cdot\nabla\textbf{u}-\nu\Delta\textbf{u}+\nabla p=0,\\
&\frac12\frac{\mathrm{d}}{\mathrm{d} t}(\mathbf{u},\mathbf{u})=-\nu\left[1+\eta(t)\right](\nabla\mathbf{u},\nabla\mathbf{u}),\\
&\nabla\cdot\textbf{u}=0.
\endaligned
\end{equation}

Use the semi-implicit first-order backward Euler method for the time discretization, we will have the following scheme 
\begin{equation}\label{ns-e3}
\aligned
&\displaystyle\frac{\textbf{u}^{n+1}-\textbf{u}^{n}}{\Delta t}+\left[1+\eta^{n+1}\right]\textbf{u}^n\cdot\nabla\textbf{u}^n-\nu\Delta\textbf{u}^{n+1}+\nabla p^{n+1}=0,\\
&\frac{(\mathbf{u}^{n+1},\mathbf{u}^{n+1})-(\mathbf{u}^{n},\mathbf{u}^{n})}{2\Delta t}=-\nu\left[1+\eta^{n+1}\right](\nabla\mathbf{u}^{n+1},\nabla\mathbf{u}^{n+1}),\\
&\nabla\cdot\textbf{u}^{n+1}=0.
\endaligned
\end{equation} 

From the second equation in \eqref{ns-e3}, one can obtain the following original dissipation law:
\begin{equation*}
\aligned
E^{n+1}-E^n=\frac{1}{2}\|\textbf{u}^{n+1}\|^2-\frac{1}{2}\|\textbf{u}^{n}\|^2=-\nu\left[1+\eta^{n+1}\right]\|\nabla\textbf{u}^{n+1}\|^2\leq0.
\endaligned
\end{equation*}

We next show how to solve the scheme \eqref{ns-e3} effectively. We can rewrite the first equation in \eqref{ns-e3} equivalently as follows:
\begin{equation}\label{ns-e6}
\aligned
(\frac{1}{\Delta t}I-\nu\Delta)\textbf{u}^{n+1}+\nabla p^{n+1}=\frac{1}{\Delta t}\textbf{u}^{n}-\textbf{u}^n\cdot\nabla\textbf{u}^n-\eta^{n+1}\textbf{u}^n\cdot\nabla\textbf{u}^n.
\endaligned
\end{equation}

Setting 
\begin{equation}\label{ns-e7}
\aligned
\textbf{u}^{n+1}=\textbf{u}_1^{n+1}+\eta^{n+1}\textbf{u}_2^{n+1},\quad p^{n+1}=p_1^{n+1}+\eta^{n+1}p_2^{n+1}.
\endaligned
\end{equation}

One can find that $\textbf{u}_1^{n+1}$ and $p_1^{n+1}$ are solutions of the following equations:
\begin{equation}\label{ns-e8}
\aligned
&(\frac{1}{\Delta t}I-\nu\Delta)\textbf{u}_1^{n+1}+\nabla p_1^{n+1}=\frac{1}{\Delta t}\textbf{u}^{n}-\textbf{u}^n\cdot\nabla\textbf{u}^n,\\
&\nabla\cdot\textbf{u}^{n+1}_1=0,
\endaligned
\end{equation}
and $\textbf{u}_2^{n+1}$ and $p_2^{n+1}$ are solutions of the following equations:
\begin{equation}\label{ns-e9}
\aligned
&(\frac{1}{\Delta t}I-\nu\Delta)\textbf{u}_2^{n+1}+\nabla p_2^{n+1}=-\textbf{u}^n\cdot\nabla\textbf{u}^n,\\
&\nabla\cdot\textbf{u}_2^{n+1}=0.
\endaligned
\end{equation}

Once $\textbf{u}_1^{n+1}$ and $\textbf{u}_2^{n+1}$ are known, we can determine $\eta^{n+1}$ by the following nonlinear algebraic equation:
\begin{equation}\label{ns-e10}
\aligned
&\frac{(\mathbf{u}_1^{n+1}+\eta^{n+1}\mathbf{u}_2^{n+1},\mathbf{u}_1^{n+1}+\eta^{n+1}\mathbf{u}_2^{n+1})-(\mathbf{u}^{n},\mathbf{u}^{n})}{2\Delta t}\\
&=-\nu\left[1+\eta^{n+1}\right](\nabla\mathbf{u}_1^{n+1}+\eta^{n+1}\nabla\mathbf{u}_2^{n+1},\nabla\mathbf{u}_1^{n+1}+\eta^{n+1}\nabla\mathbf{u}_2^{n+1}).
\endaligned
\end{equation}

To save the computational costs and compute $\eta^{n+1}$ effectively, we can also use the similar technique as before.  A highly efficient algorithm to compute $\textbf{u}^{n+1}$, $p^{n+1}$ and $\eta^{n+1}$ is as follows:

\begin{scheme}

\textbf{Step I}: Calculate the intermediate solution $\overline{\textbf{u}}^{n+1}$ and $\overline{p}^{n+1}$ from the following first-order scheme:
\begin{equation}\label{ns-e11}
   \begin{array}{l}
\displaystyle\frac{\overline{\textbf{u}}^{n+1}-\textbf{u}^{n}}{\Delta t}+\textbf{u}^n\cdot\nabla\textbf{u}^n-\nu\Delta\overline{\textbf{u}}^{n+1}+\nabla \overline{p}^{n+1}=0,\\
\nabla\cdot\textbf{u}^{n+1}=0.
   \end{array}
  \end{equation}

\textbf{Step II}: Compute the Lagrange multiplier variable $\eta^{n+1}$ as follows:
 \begin{equation}\label{nds-e12}
\eta^{n+1}=\left\{
   \begin{array}{ll}
0,\quad E(\overline{\textbf{u}}^{n+1})\leq E(\textbf{u}^n),\\
\overline{\eta}^{n+1},\quad E(\overline{\textbf{u}}^{n+1})>E(\textbf{u}^n).
   \end{array}
   \right.
\end{equation} 
where $\overline{\eta}^{n+1}$ can be obtained as follows:
\begin{equation}\label{nds-e13}
\aligned
&\frac{(\overline{\textbf{u}}^{n+1}+\eta^{n+1}\mathbf{u}_2^{n+1},\overline{\textbf{u}}^{n+1}+\eta^{n+1}\mathbf{u}_2^{n+1})-(\mathbf{u}^{n},\mathbf{u}^{n})}{2\Delta t}\\
&=-\nu\left[1+\eta^{n+1}\right](\nabla\overline{\textbf{u}}^{n+1}+\eta^{n+1}\nabla\mathbf{u}_2^{n+1},\nabla\overline{\textbf{u}}^{n+1}+\eta^{n+1}\nabla\mathbf{u}_2^{n+1}).
\endaligned
\end{equation}
where $\textbf{u}_2^{n+1}$ can be solved directly as follows:
\begin{equation}\label{nds-e14}
\aligned 
&(\frac{1}{\Delta t}I-\nu\Delta)\textbf{u}_2^{n+1}+\nabla p_2^{n+1}=-\textbf{u}^n\cdot\nabla\textbf{u}^n,\\
&\nabla\cdot\textbf{u}_2^{n+1}=0.
\endaligned
\end{equation}

\textbf{Step III}: Update $\textbf{u}^{n+1}$ and $p^{n+1}$ as
\begin{equation}\label{nds-e15}
\textbf{u}^{n+1}=\overline{\textbf{u}}^{n+1}+\eta^{n+1}\textbf{u}_2^{n+1},\quad p^{n+1}=\overline{p}^{n+1}+\eta^{n+1}p_2^{n+1}.
\end{equation}  
\end{scheme}
\section{Examples and discussion}
In this section, we provide some numerical experiments to verify our theoretical results of the constructed schemes.
In all of the examples, periodic boundary conditions are considered and a Fourier spectral method is applied.
\subsection{Accuracy and energy stability test}
%\begin{example}
In this expample, we first verify the accuracy of the proposed numerical schemes for the Allen–Cahn ($\mathcal{G}= I$) and
the Cahn–Hillard equation ($\mathcal{G}=-\Delta$) in the following form:
\begin{equation}\label{Q1}
	\begin{array}{l}
		\displaystyle\frac{\partial \phi}{\partial t}=- M\mathcal{G}\mu+f,\\
		\mu=-\Delta\phi+F'(\phi),
	\end{array}
\end{equation}
Here, $M$ is a mobility constant and  $F'(\phi)=\frac{1}{\epsilon^2}\phi(\phi^2-1)$.
Model parameter values are 
\begin{equation*}
 M=1,\quad \epsilon=1,
\end{equation*}
and solve \eqref{Q1}
with right hand sides $f$ chosen so that the exact solution is
\begin{equation*}
		\phi(x,y,t)=\exp(-t)\cos(\pi x)\cos(\pi y).
\end{equation*}

To discretize the spatial variables, we define the domain as $\Omega= [0, 2\pi)^2$ and use $256\times256$ modes, so the spatial discretization
error is negligible compared to the time discretization.  In Fig. \ref{error1}, we list the $L^\infty$ errors between the numerical solution and the exact solution at $T=1$. It can be observed that all schemes achieve the expected accuracy in time.
\begin{figure}[!h]
	\centering
	\subfigure[]{
		\begin{minipage}[c]{0.4\textwidth}
			\includegraphics[width=1\textwidth]{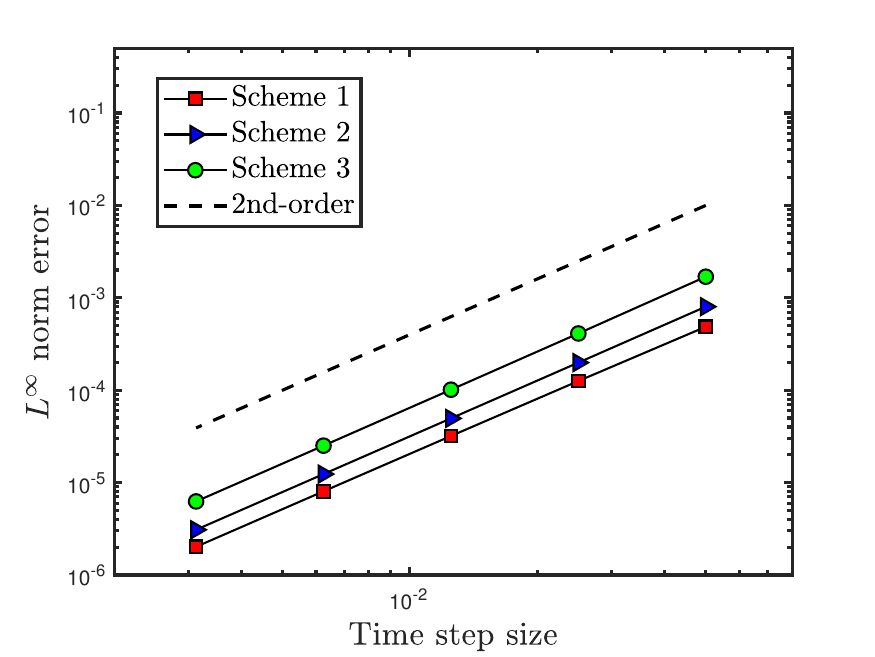}
		\end{minipage}
	}
	\subfigure[]{
		\begin{minipage}[c]{0.4\textwidth}
			\includegraphics[width=1\textwidth]{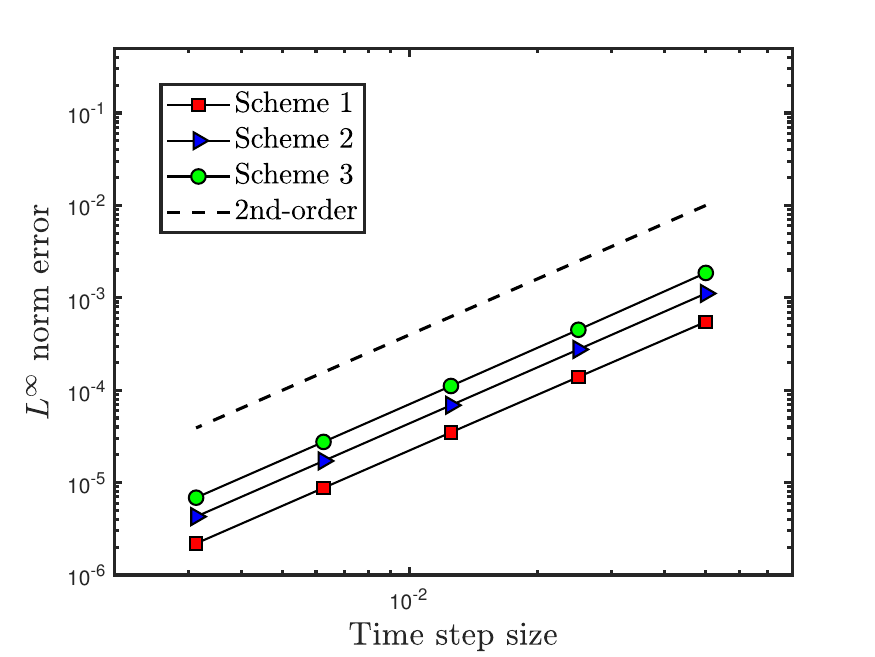}
		\end{minipage}
	}
	{\caption{Numerical convergence rate test: with given exact solution for Allen–Cahn equation (left) and Cahn–Hilliard equation (right).}\label{error1}}
\end{figure}

Next, we investigate the energy stability of the numerical schemes.
Initially, the right-hand side $f$ is set to 0, and we start with a random condition given by:
\begin{equation*}
	\phi(x,y,0)=0.03+0.001\text{rand}(x,y),
\end{equation*}
in which $\text{rand}(x,y)$  represents a uniformly distributed random function  in the domain $[-1, 1]^2$.
In Fig. \ref{AC} and Fig. \ref{CH}, both Allen–Cahn equation and Cahn–Hilliard equation are plotted with a model parameter of  $\epsilon^2= 0.005$.
As observed from the Fig. \ref{AC1}, it can be observed that the computed energy decays over time for both Scheme 1 and Scheme 2. However, the Scheme 2, which is proposed in this paper, exhibits higher computational efficiency  compared to Scheme 1, as depicted in Fig. \ref{AC2}.
As for the Cahn–Hilliard equation shown in Fig. \ref{CH}, Scheme 2 requires a smaller time step compared to Scheme 1 in \cite{cheng2020new},  leading to greater computational expenses.

\begin{figure}[!h]
	\centering
	\subfigure[Energy evolution]{
		\begin{minipage}[c]{0.4\textwidth}
			\includegraphics[width=1\textwidth]{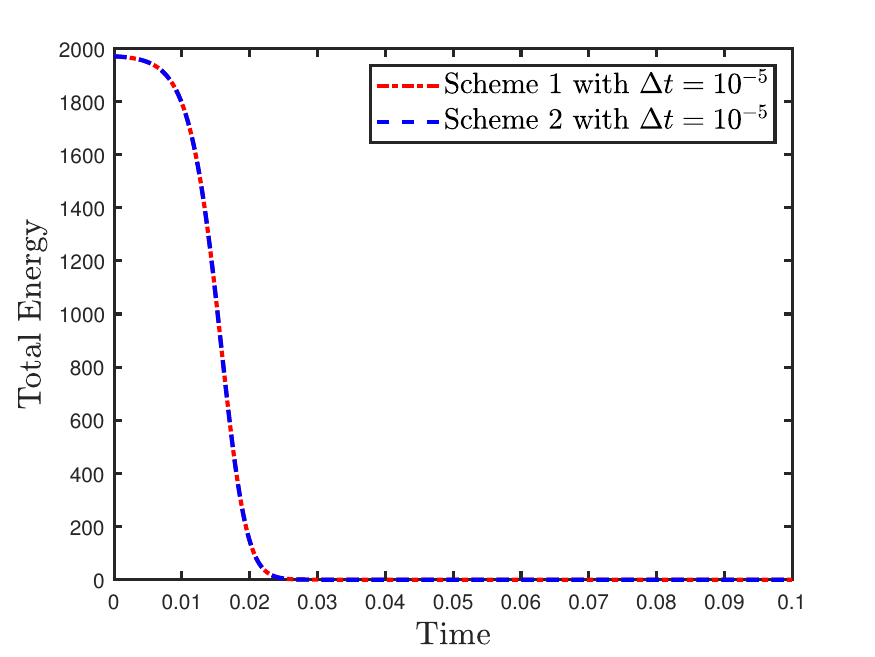}\label{AC1}
		\end{minipage}
	}
	\subfigure[CPU time]{
		\begin{minipage}[c]{0.4\textwidth}
			\includegraphics[width=1\textwidth]{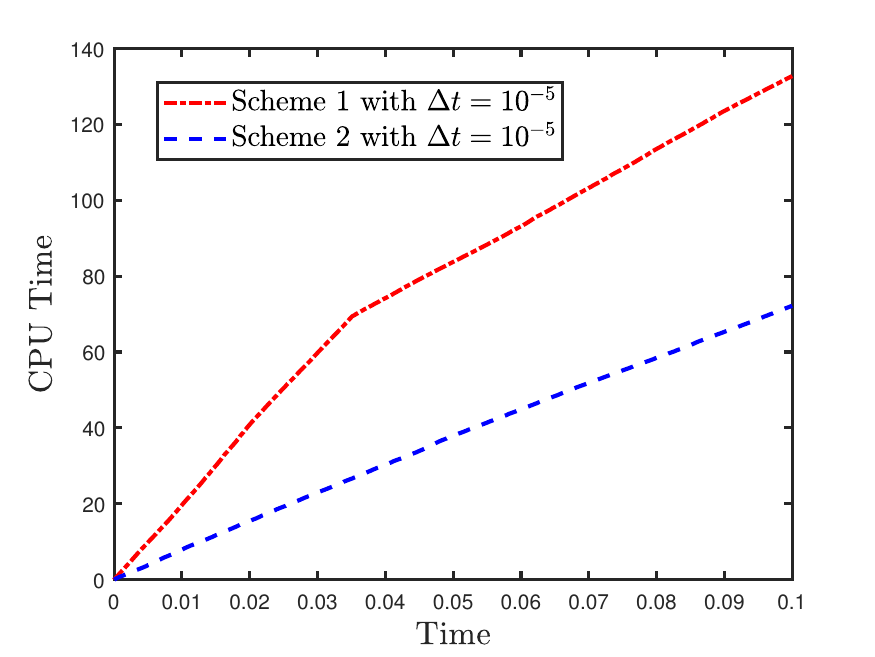}\label{AC2}
		\end{minipage}
	}
	{\caption{Comparison of the Allen-Cahn equation with $M=1$ and $\epsilon^2= 0.005$.}\label{AC}}
\end{figure}

\begin{figure}[!h]
	\centering
	\subfigure[Energy evolution]{
		\begin{minipage}[c]{0.4\textwidth}
			\includegraphics[width=1\textwidth]{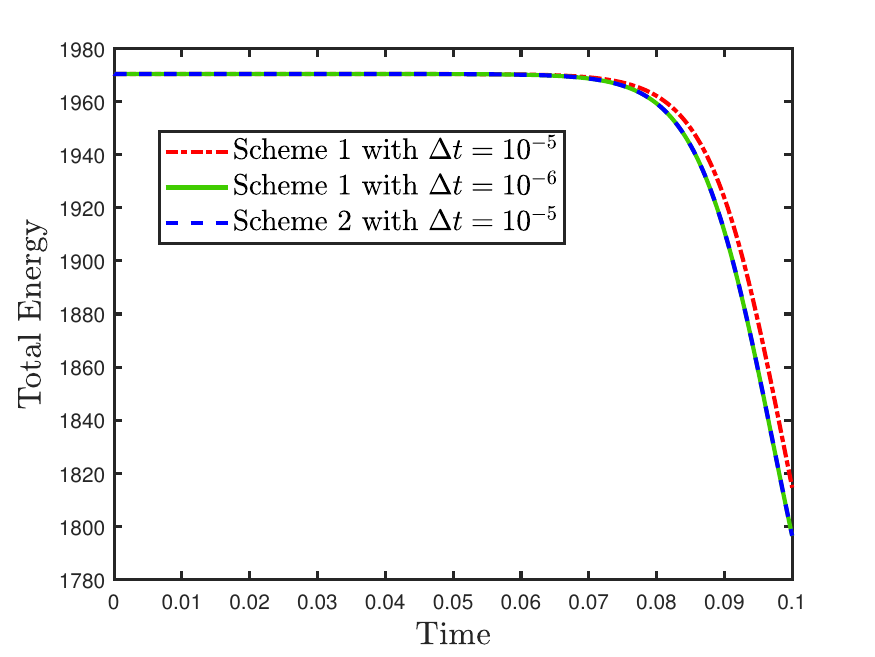}\label{CH1}
		\end{minipage}
	}
	\subfigure[CPU time]{
		\begin{minipage}[c]{0.4\textwidth}
			\includegraphics[width=1\textwidth]{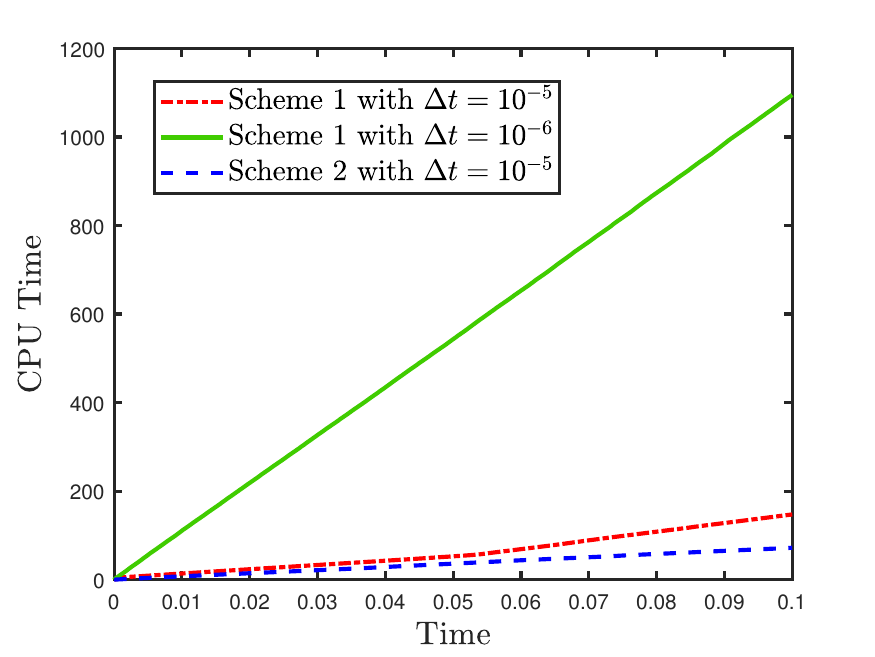}\label{CH2}
		\end{minipage}
	}
	{\caption{Comparison of the Cahn-Hilliard equation with $M=0.01$ and $\epsilon^2= 0.005$.}\label{CH}}
\end{figure}
%\end{example}

\subsection{Molecular beam epitaxial without slope selection}
We consider the molecular beam epitaxial (MBE) model without slope selection \cite{schwoebel1966step,villain1991continuum} as an example, where the nonlinear functional is unbounded from below. The standard no-slope MBE model is a fourth-order parabolic equation, known as the $L^2$ gradient flow of the Ehrlich-Schwoebel energy functional, which can be expressed as:
\begin{equation*}
	E(\phi)=\int_\Omega \frac{\epsilon^2}{2}|\Delta \phi|^2+F(\phi)d\textbf{x},
\end{equation*}
where $F(\phi)=-\frac{1}{2}\ln(1+|\nabla\phi|^2)$.
Then the MBE equation can be written as follows
\begin{equation*}
		\displaystyle\frac{\partial \phi}{\partial t}=-M\frac{\delta E(\phi)}{\delta \phi}= -M(\epsilon^2\Delta^2\phi+\nabla\cdot \textbf{f}(\nabla\phi)),\\
\end{equation*}
 The mobility constant is denoted as $M$, and the nonlinear force vector  is defined as $\textbf{f}(\textbf{v}):=\frac{\textbf{v}}{1+|\textbf{v}|^2}$.

Similar to section \ref{s3.1}, we can easily construct unconditionally stable numerical schemes for the MBE equation without slope selection, which can be implemented efficiently.
In Fig. \ref{MBEEeta}, we present the time evolutions of the total energy and the Lagrange multiplier $\eta$ in the domain $[0, 2\pi]^2$.  The parameters used are $M=0.1$ and $\epsilon=0.03$, with the time step of $\Delta t=5.2\times10^{-3}$. The initial condition is given by:
\begin{equation*}
	\phi(x,y,0)=0.01\text{rand}(x,y),
\end{equation*}
where $\text{rand}(x,y)$ represents random data between $[-1, 1]^2$.
Fig.  \ref{MBE} shows the isolines of the numerical solutions of the height function $\phi$ and its Laplacian $\Delta \phi$ for the model without slope selection.
\begin{figure}[!h]
	\centering
	\subfigure[]{
		\begin{minipage}[c]{0.4\textwidth}
			\includegraphics[width=1\textwidth]{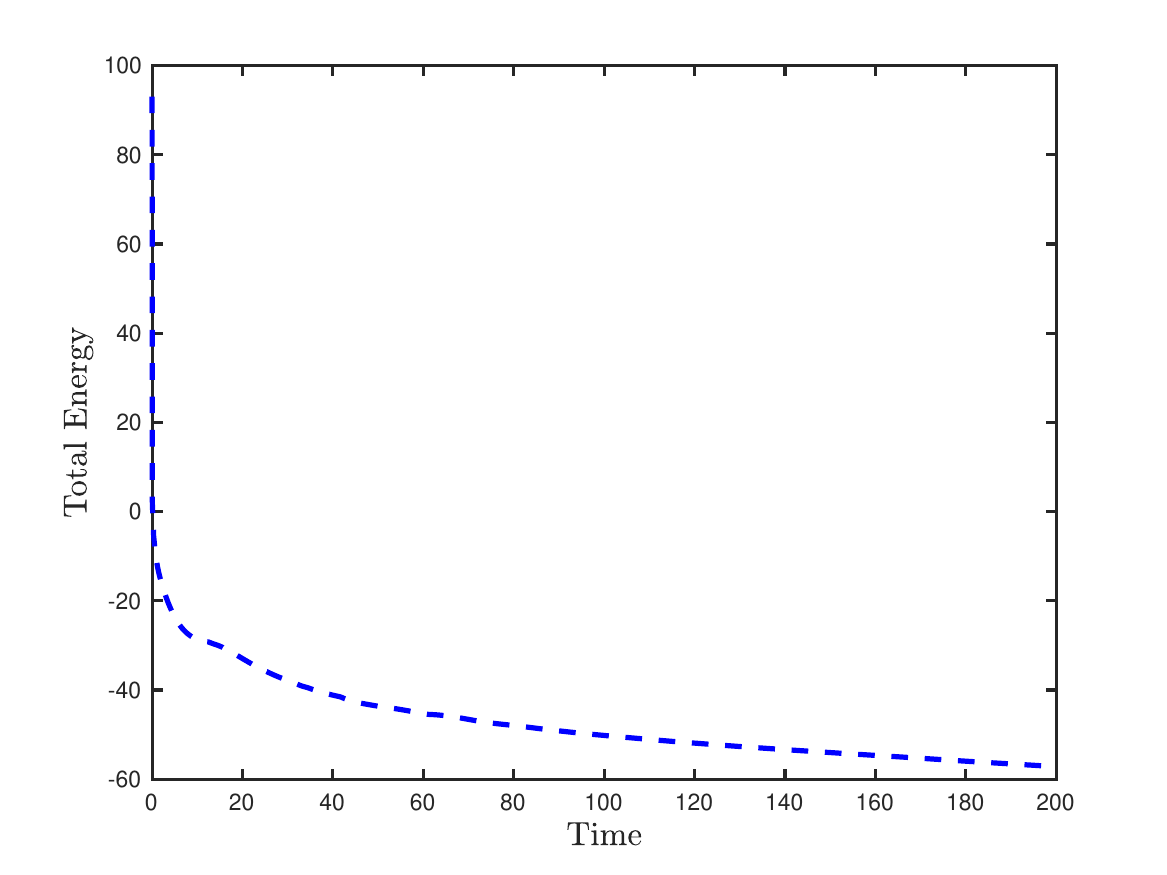}
		\end{minipage}
	}
	\subfigure[]{
		\begin{minipage}[c]{0.4\textwidth}
			\includegraphics[width=1\textwidth]{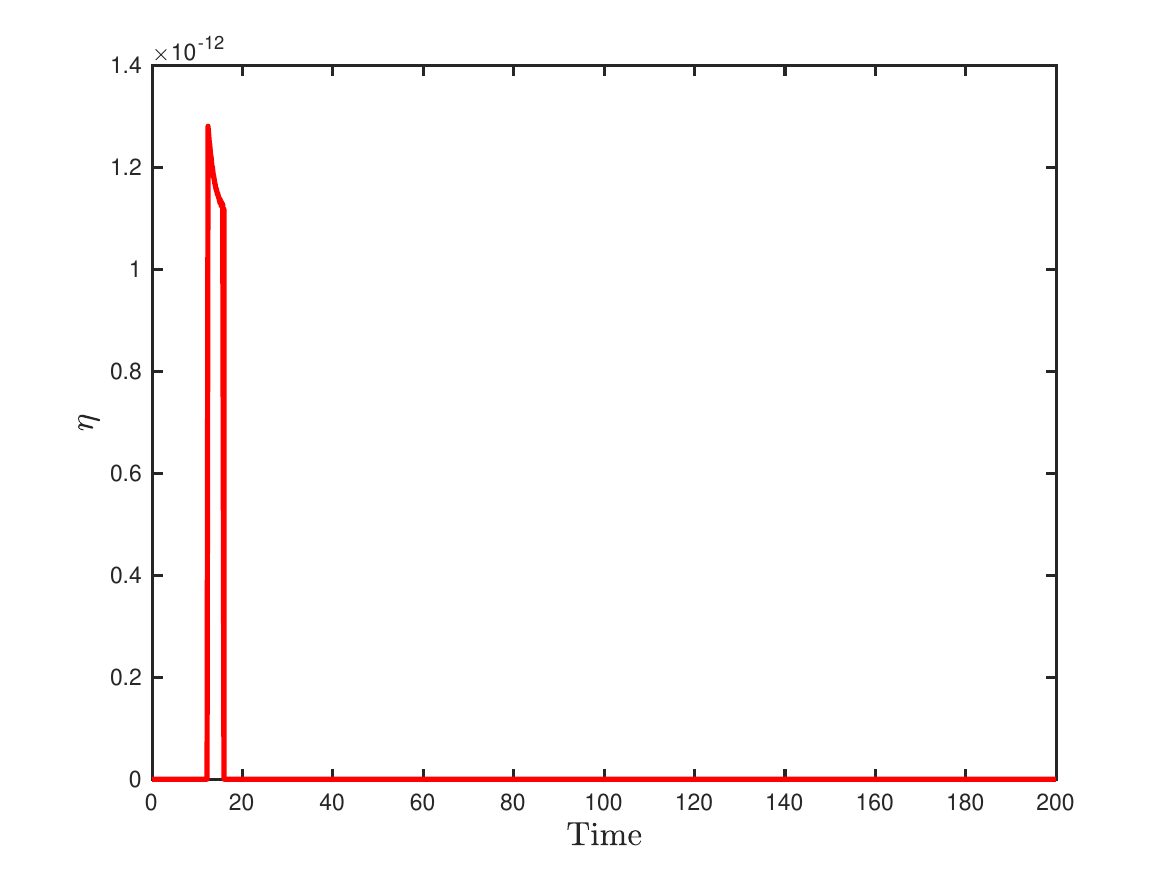}
		\end{minipage}
	}
	{\caption{The temporal  evolution of energy and the Lagrange multiplier $eta$ for the MBE model without slope selection.}\label{MBEEeta}}
\end{figure}

\begin{figure}[H]
	\centering
	\subfigure[T=2]{
		\begin{minipage}[c]{0.4\textwidth}
			\includegraphics[width=1\textwidth]{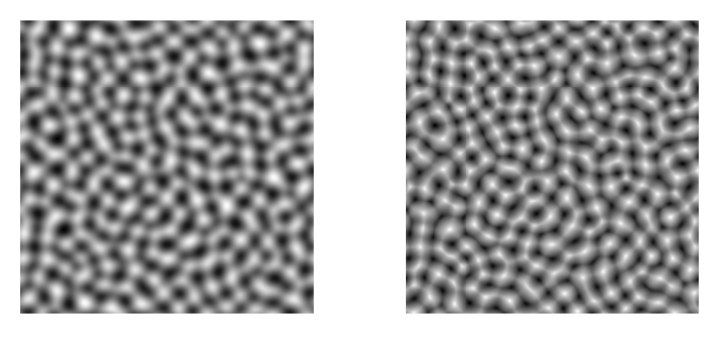}
		\end{minipage}
	}
	\subfigure[T=5]{
		\begin{minipage}[c]{0.4\textwidth}
			\includegraphics[width=1\textwidth]{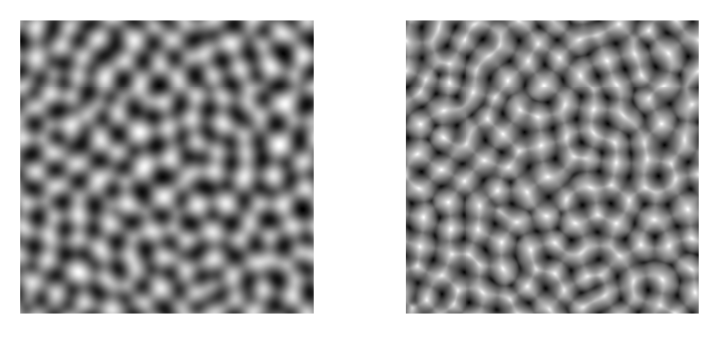}
		\end{minipage}
	}
	\subfigure[T=10]{
	\begin{minipage}[c]{0.4\textwidth}
		\includegraphics[width=1\textwidth]{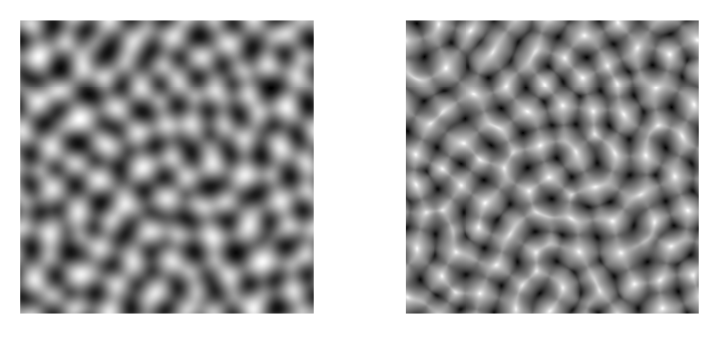}
	\end{minipage}
}
\subfigure[T=30]{
	\begin{minipage}[c]{0.4\textwidth}
		\includegraphics[width=1\textwidth]{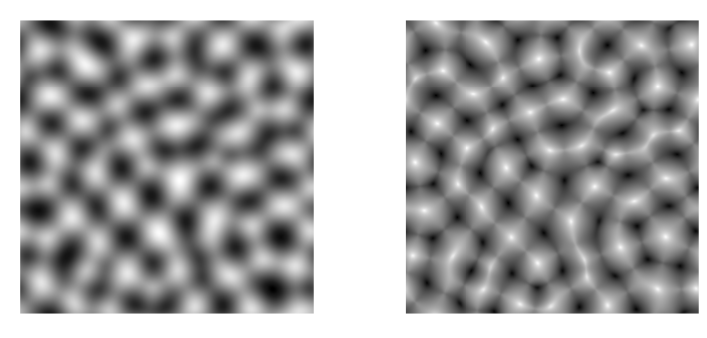}
	\end{minipage}
}
	\subfigure[T=50]{
	\begin{minipage}[c]{0.4\textwidth}
		\includegraphics[width=1\textwidth]{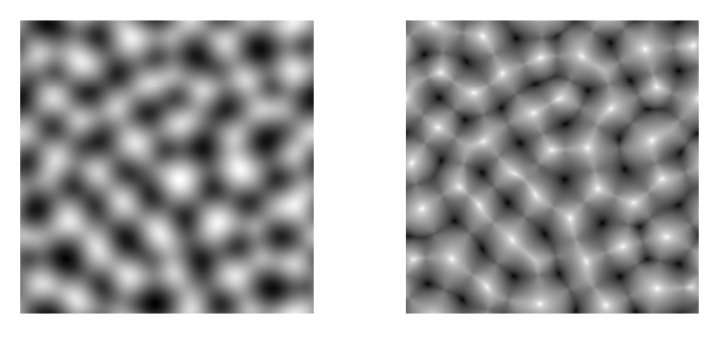}
	\end{minipage}
}
\subfigure[T=200]{
	\begin{minipage}[c]{0.4\textwidth}
		\includegraphics[width=1\textwidth]{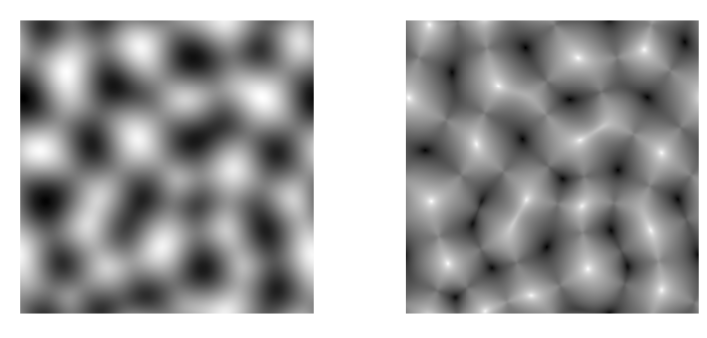}
	\end{minipage}
}
	{\caption{The left subfigure represents $\phi$ and the right subfigure represents $\Delta \phi$. Each snapshot is taken at $T = 2, 5, 10, 30, 50, 200$.}\label{MBE}}
\end{figure}

Furthermore, we perform numerical simulations of coarsening dynamics in 3D by assigning a random number to each grid point ranging from $-0.001$ to $0.001$ as the initial condition.
The simulations are conducted in the domain $[0, L]^3$ with $L = \pi$.
The space is discretized using $128 \times 128 \times 128$ grid points, and the time step is set to $\Delta t=4.3\times10^{-3}$.
Fig. \ref{MBEEeta3D} shows the energy evolution and the Lagrange multiplier $\eta$ for the MBE model without slope selection in 3D.
Fig. \ref{MBE3D} displays three isosurfaces for $\phi=-0.025,0$ and $0.025$, colored in blue, green, and red, respectively.

\begin{figure}[H]
	\centering
	\subfigure[]{
		\begin{minipage}[c]{0.4\textwidth}
			\includegraphics[width=1\textwidth]{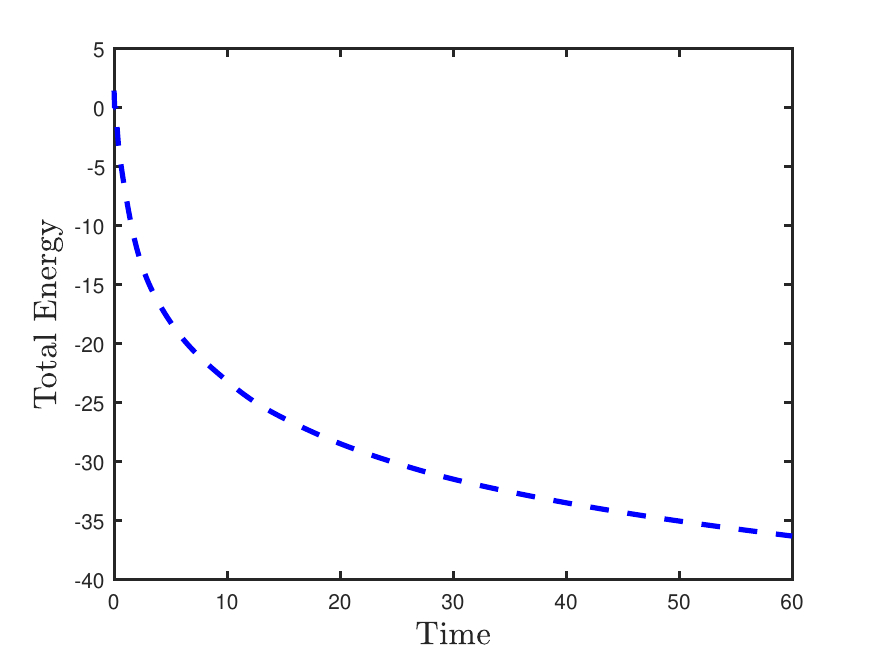}
		\end{minipage}
	}
	\subfigure[]{
		\begin{minipage}[c]{0.4\textwidth}
			\includegraphics[width=1\textwidth]{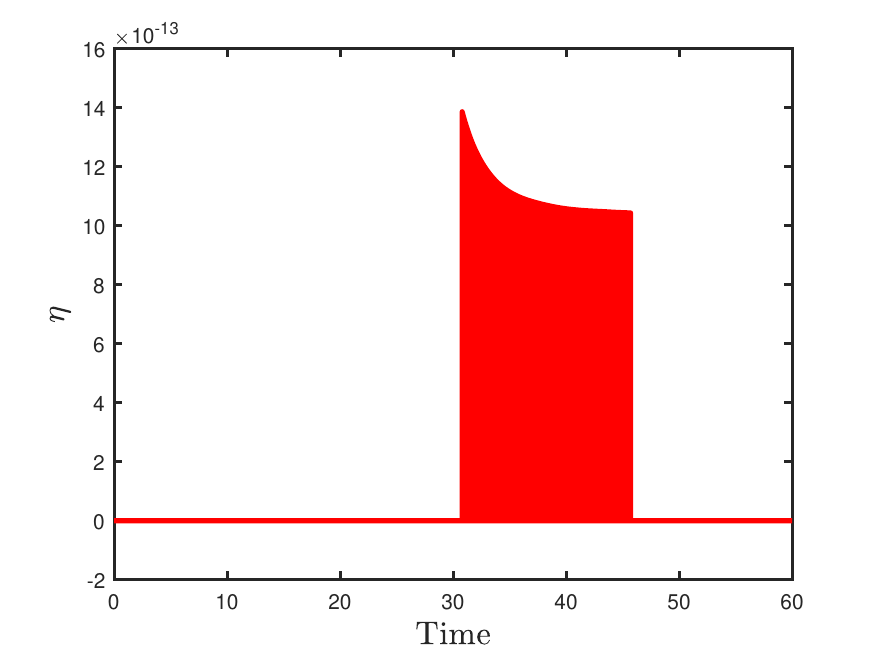}
		\end{minipage}
	}
	{\caption{The temporal evolution of energy and the Lagrange multiplier $eta$ for 3D MBE model without slope selection.}\label{MBEEeta3D}}
\end{figure}

\begin{figure}[!h]
	\centering
	\subfigure[T=1]{
		\begin{minipage}[c]{0.3\textwidth}
			\includegraphics[width=1\textwidth]{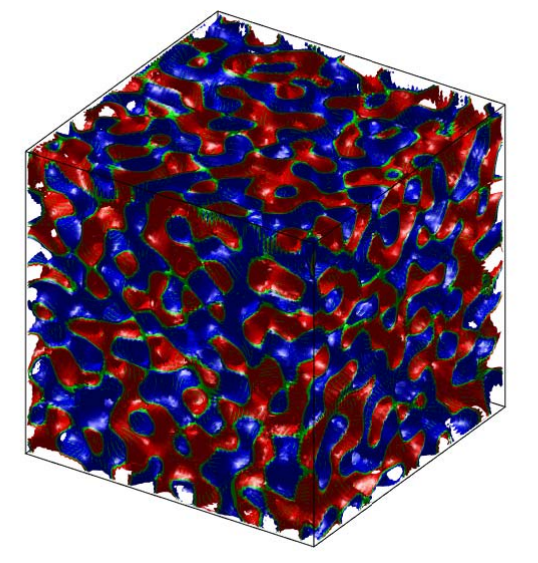}
		\end{minipage}
	}
	\subfigure[T=5]{
		\begin{minipage}[c]{0.3\textwidth}
			\includegraphics[width=1\textwidth]{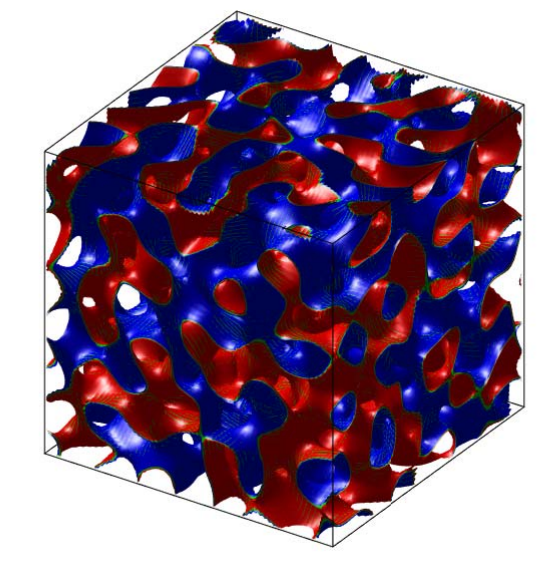}
		\end{minipage}
	}
	\subfigure[T=10]{
		\begin{minipage}[c]{0.3\textwidth}
			\includegraphics[width=1\textwidth]{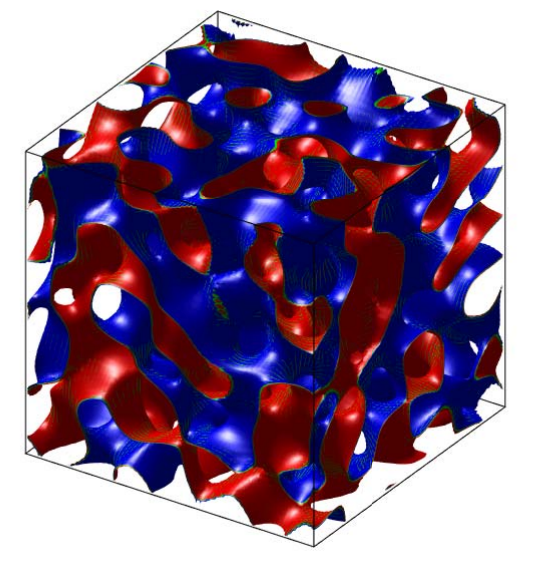}
		\end{minipage}
	}
	\subfigure[T=20]{
		\begin{minipage}[c]{0.3\textwidth}
			\includegraphics[width=1\textwidth]{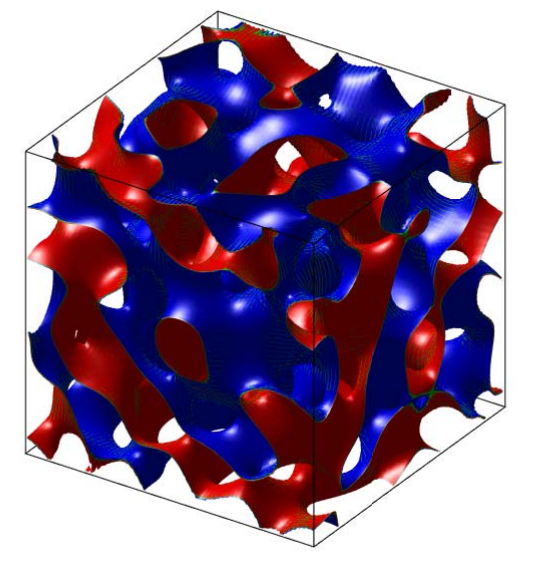}
		\end{minipage}
	}
	\subfigure[T=40]{
		\begin{minipage}[c]{0.3\textwidth}
			\includegraphics[width=1\textwidth]{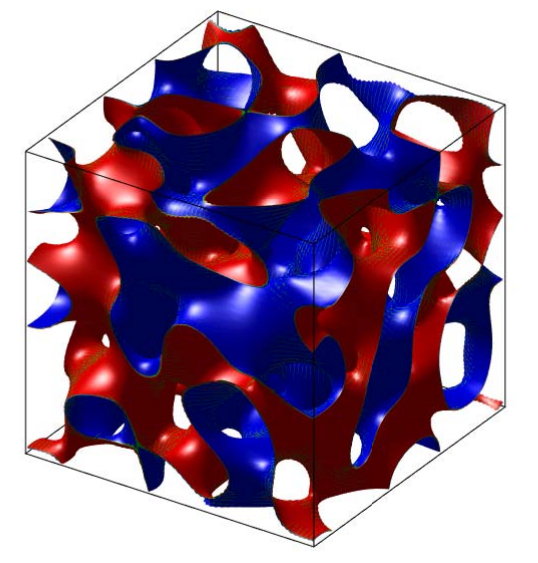}
		\end{minipage}
	}
	\subfigure[T=60]{
		\begin{minipage}[c]{0.3\textwidth}
			\includegraphics[width=1\textwidth]{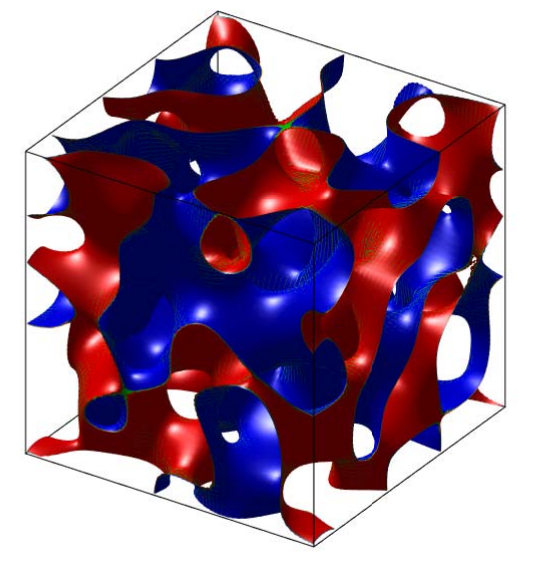}
		\end{minipage}
	}
	{\caption{The three isosurfaces of numerical solutions of $\phi= -0.025, 0$ and $0.025$ of 3D coarsening dynamics. Each snapshot is taken at $T = 1, 5, 10, 20, 40, 60$. }\label{MBE3D}}
\end{figure}

%\end{example}

%\begin{example}
\subsection{Ternary Cahn-Hilliard phase-field model}
The numerical approximation of the three-phase Cahn-Hilliard system  presents a significant challenge due to its nonlinear term.
In this subsection, we will solve the three-phase Cahn-Hilliard phase-field model using a semi-implicit approach with a Lagrange multiplier, as developed by Boyer et al. \cite{hong2023high,boyer2006study,boyer2011numerical}.
	
The model can be expressed as follows. The incompressibility condition	
links the three variables $\phi_1$, $\phi_2$, and $\phi_3$:
	$$
	\phi_1+\phi_2+\phi_3=1,
	$$
and the three-phase free energy is given by a specific expression.
\begin{equation}\label{energy}
	\mathcal{E}\left[\phi_1, \phi_2, \phi_3\right]=\frac{3 \epsilon^2}{8} \sum_{l=1}^3 \int_{\Omega} \Sigma_l\left|\nabla \phi_l\right|^2 \mathrm{~d} \mathbf{x}+12 \int_{\Omega} F\left(\phi_1, \phi_2, \phi_3\right) \mathrm{d} \mathbf{x},
\end{equation}
%where the coefficient of entropic term $\Sigma_l$ can be negative for some specific situations. To be algebraically consisten with the two-phase system, the three surface tension parameters $\sigma_{12}, \sigma_{13}$ and $\sigma_{23}$ should satisfy the following conditions
The surface tension parameters $\sigma_{12}, \sigma_{13}$ and $\sigma_{23}$ satisfy certain requirements to maintain consistency with the two-phase system algebraically
$$
\Sigma_1=\sigma_{12}+\sigma_{13}-\sigma_{23}, \quad \Sigma_2=\sigma_{12}+\sigma_{23}-\sigma_{13}, \quad \Sigma_3=\sigma_{13}+\sigma_{23}-\sigma_{12} .
$$

%The nonlinear potential $F\left(\phi_1, \phi_2, \phi_3\right)$ is given by
%	$$
%	F\left(\phi_1, \phi_2, \phi_3\right)=\frac{\Sigma_1}{2} \phi_1^2\left(1-\phi_1\right)^2+\frac{\Sigma_2}{2} \phi_2^2\left(1-\phi_2\right)^2+\frac{\Sigma_3}{2} \phi_3^2\left(1-\phi_3\right)^2+3 \Lambda \phi_1^2 \phi_2^2 \phi_3^2,
%	$$
%	where $\Lambda$ is a non-negative constant.
The volume conservation constraint $\phi_3=1-\phi_1-\phi_2$ allows us to recast the energy functional \eqref{energy} as follows.
	$$
	\mathcal{E}\left[\phi_1, \phi_2\right]=\frac{3 \epsilon^2}{8} \int_{\Omega}\left(\Sigma_1\left|\nabla \phi_1\right|^2+\Sigma_2\left|\nabla \phi_2\right|^2+\Sigma_3\left|\nabla \phi_1+\nabla \phi_2\right|^2\right) \mathrm{d} \mathbf{x}+12 \int_{\Omega} F\left(\phi_1, \phi_2\right) \mathrm{d} \mathbf{x},
	$$
where $F\left(\phi_1, \phi_2\right)$ is given by
	$$
	F\left(\phi_1, \phi_2\right)=\frac{\Sigma_1}{2} \phi_1^2\left(1-\phi_1\right)^2+\frac{\Sigma_2}{2} \phi_2^2\left(1-\phi_2\right)^2+\frac{\Sigma_3}{2}\left(\phi_1+\phi_2\right)^2\left(1-\phi_1-\phi_2\right)^2+3 \Lambda \phi_1^2 \phi_2^2\left(1-\phi_1-\phi_2\right)^2 .
	$$
	Here, $\Lambda$ is a non-negative constant.
%We note that coefficient $\Sigma_i$ is called the spreading coefficient (please refer to [63,65] and the references therein) of phase $i$ at the interface between phase $j$ and phase $k$. If $\Sigma_i>0$, the spreading is said to be "partial", and if one of $\Sigma_i<0$, it is called "total". 
%For the total spreading case, the following conditions are assumed to ensure the well-posedness of system
%	$$
%	\Sigma_1 \Sigma_2+\Sigma_1 \Sigma_3+\Sigma_2 \Sigma_3>0, \quad \Sigma_i+\Sigma_j>0, \quad \forall i \neq j
%	$$
The coupled Cahn-Hilliard model defines the dynamic equation as follows.
\begin{equation}\label{cch}
	\begin{aligned}
		& \partial_t \phi_l=M \Delta \frac{\mu_l}{\Sigma_l}, \quad l=1,2, \\
		& \mu_1=-\frac{3 \epsilon^2}{4}\left(\Sigma_1+\Sigma_3\right) \Delta \phi_1-\frac{3 \epsilon^2}{4} \Sigma_3 \Delta \phi_2+12 \frac{\partial F\left(\phi_1, \phi_2\right)}{\partial \phi_1}, \\
		& \mu_2=-\frac{3 \epsilon^2}{4} \Sigma_3 \Delta \phi_1-\frac{3 \epsilon^2}{4}\left(\Sigma_2+\Sigma_3\right) \Delta \phi_2+12 \frac{\partial F\left(\phi_1, \phi_2\right)}{\partial \phi_2},
	\end{aligned}
\end{equation}
	where the initial conditions are given by
	$$
	\left.\phi_l(\mathbf{x}, t)\right|_{t=0}=\phi_l^0(\mathbf{x}), \quad l=1,2, \quad \phi_3^0=1-\phi_1^0(\mathbf{x})-\phi_2^0(\mathbf{x}) .
	$$
By taking the inner product of the first equation in \eqref{cch} with $\mu_1$, and $\mu_2$, the second equation in \eqref{cch} with $-\partial_t \phi_1$, the third equation in \eqref{cch} with $-\partial_t \phi_2$, we obtain immediately the energy dissipation law:
	$$
	\frac{\mathrm{d} \mathcal{E}}{\mathrm{d} t}=-M\left(\frac{1}{\Sigma_1}\left\|\nabla \mu_1\right\|^2+\frac{1}{\Sigma_2}\left\|\nabla \mu_2\right\|^2\right) \leq 0 .
	$$
Similarly, unconditionally stable numerical schemes can be easily constructed and efficiently implemented as follows:
\begin{scheme}
\begin{equation}\label{cch1}
	\begin{aligned}
		 \frac{\phi_l^{n+1}-\phi_l^n}{\Delta t}= \frac{M}{\Sigma_l}\Delta \mu_l^{n+\frac{1}{2}}, \quad l=1,2,& \\
		 \mu_1^{n+\frac{1}{2}}=-\frac{3 \epsilon^2}{4}\left(\Sigma_1+\Sigma_3\right) \Delta \phi_1^{n+\frac{1}{2}}-\frac{3 \epsilon^2}{4} &\Sigma_3 \Delta \phi_2^{n+\frac{1}{2}}+12(1+\eta^{n+\frac12})\frac{\partial F}{\partial \phi_1}\left(\phi_1^{*,n+\frac12}, \phi_2^{*,n+\frac12}\right), \\
		 \mu_2^{n+\frac{1}{2}}=-\frac{3 \epsilon^2}{4} \Sigma_3 \Delta \phi_1^{n+\frac{1}{2}}
		 -\frac{3 \epsilon^2}{4}(\Sigma_2+\Sigma_3)& \Delta \phi_2^{n+\frac{1}{2}}+12(1+\eta^{n+\frac12}) \frac{\partial F}{\partial \phi_2}\left(\phi_1^{*,n+\frac12}, \phi_2^{*,n+\frac12}\right),\\
		(F(\phi_1^{n+1},\phi_2^{n+1}),1)-(F(\phi_1^{n},\phi_2^{n}),1)=
&(1+\eta^{n+\frac12})\left(\frac{\partial F}{\partial \phi_1}\left(\phi_1^{*,n+\frac12}, \phi_2^{*,n+\frac12}\right),\phi_1^{n+1}-\phi_1^{n}\right)\\
&+(1+\eta^{n+\frac12}) \left(\frac{\partial F}{\partial \phi_2}\left(\phi_1^{*,n+\frac12},\phi_2^{*,n+\frac12}\right),\phi_2^{n+1}-\phi_2^{n}\right),
	\end{aligned}
\end{equation}
where ${\phi}_1^{*,n+\frac12}=\frac32\phi_1^n-\frac12\phi_1^{n-1}$, \ ${\phi}_2^{*,n+\frac12}=\frac32\phi_2^n-\frac12\phi_2^{n-1}$.
\end{scheme}

\begin{example}(Accuracy and energy stability test) To demonstrate the accuracy and energy stability of the ternary Cahn-Hilliard system, we present an example with the following specific initial conditions:
	$$
	\begin{aligned}
		& \phi_i^0(x, y)=\frac{1}{2}\left(1+\tanh \left(\frac{r_i-\sqrt{\left(x-x_i\right)^2+\left(y-y_i\right)^2}}{\epsilon}\right)\right), \quad i=1,2, \\
		& \phi_3^0(x, y)=1-\phi_1^0(x, y)-\phi_2^0(x, y),
	\end{aligned}
	$$
where $r_1=r_2=0.35,\ x_1=1.37,\ x_2=0.63$ and $y_1=y_2=1.0$. 
The computational domain is $\Omega=[0,2] \times[0,2]$ 
and we use $256\times256$ modes to discretize the space variables.
The coefficients $M=10^{-5},\ \epsilon=0.02$, and $\Lambda=7$ are chosen.
The corresponding $L^\infty$ errors computed by Scheme 2 at $T=0.2$ are 
summarized in Fig. \ref{TCHo}, with a reference solution using $\Delta t=1.0 \times 10^{-5}$. Fig. \ref{TCHee} demonstrates the unconditional stability of the scheme with $\Delta t=1.0 \times 10^{-3}$.

\begin{figure}[htp]
	\centering
	\subfigure[]{
		\begin{minipage}[c]{0.4\textwidth}
			\includegraphics[width=1\textwidth]{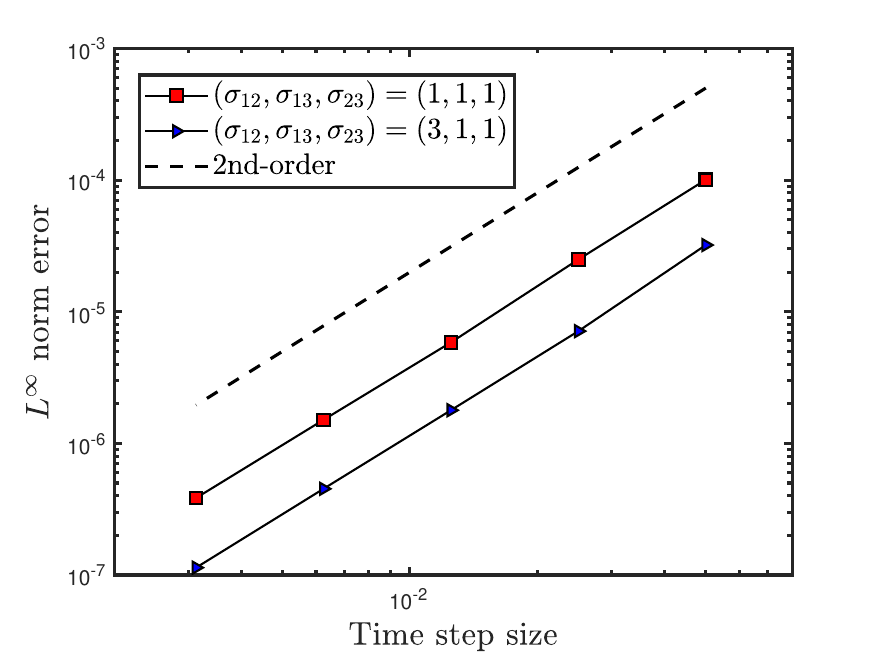}\label{TCHo}
		\end{minipage}
	}
	\subfigure[]{
		\begin{minipage}[c]{0.4\textwidth}
			\includegraphics[width=1\textwidth]{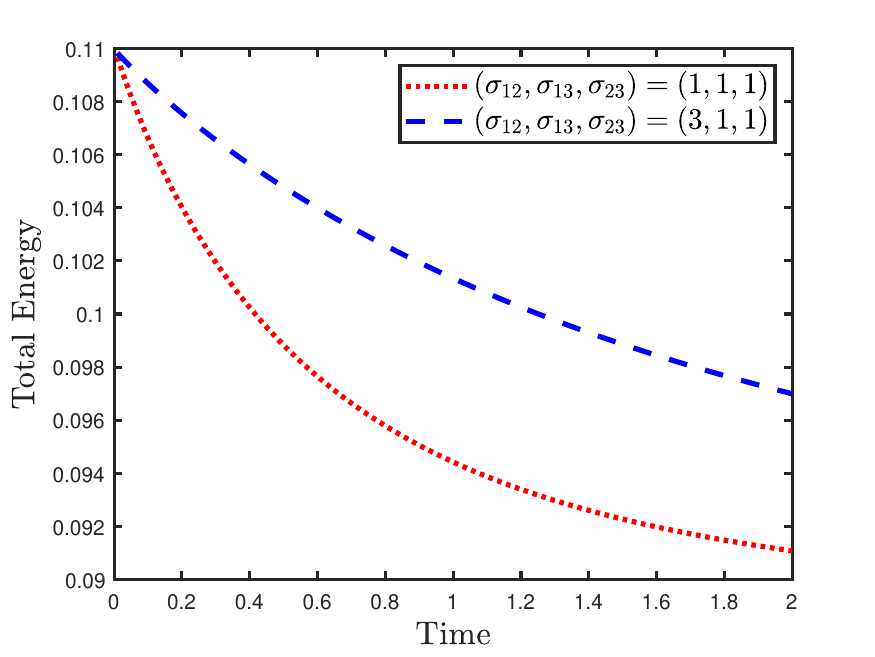}\label{TCHee}
		\end{minipage}
	}
	{\caption{Numerical convergence rate (left) and energy stability test (right) for the ternary Cahn–Hilliard phase-field model.}\label{TCHO2}}
\end{figure}

%	The errors are calculated the same way as we have done in the previous example. Here, surface parameter values $\left(\sigma_{12}, \sigma_{13}, \sigma_{23}\right)=(1,1,1)$ (for the partial spreading) and $\left(\sigma_{12}, \sigma_{13}, \sigma_{23}\right)=(3,1,1)$ (for the total spreading) are used, respectively. The SVM Crank-Nicolson type (denoted by SVM-CN) scheme (Ref. [2]) and the newly developed high-order schemes, i.e., SVM-RK4 and SVM-RK6, are compared. The temporal convergence rates in the discrete $L^2$ norm (i.e., $\frac{1}{3} \sum_{i=1}^3\left\|e\left(\phi_i\right)\right\|$ ) and $L^{\infty}$ norm (i.e., $\max _i\left\|e\left(\phi_i\right)\right\|_{\infty}$ ) at time $t_f=0.2$ are displayed in Fig. 5. We observe that the two newly proposed schemes asymptotically match the expected convergence order in time. In particular, the high-order SVM-RK schemes are significantly smaller than the second-order SVM-CN in several orders of magnitude.
%	
%	In another set of numerical experiments, we obtain discrete $L^2$ errors for the phase-field variable $\phi_i(i=1,2,3)$ at $t=1$ with values smaller than $1.0 \times 10^{-8}$ while choosing $\tau=1.0 \times 10^{-4}$ for the second-order SVM-CN scheme, $\tau=1.0 \times 10^{-2}$ for the SVM-RK4 scheme and $\tau=4 \times 10^{-2}$ for the SVM-RK6. The total CPU costs in the three numerical experiments are presented as a bar graph in Fig. 6, where the corresponding discrete $L^2$ errors are also marked in the graph. We observe that the high-order SVM-RK schemes cost much less CPU time than the
\end{example}

\begin{example}(Spinodal decomposition)
Furthermore, we consider an example of phase separation (or spinodal decomposition), where an initial homogeneous three-phase mixture evolves into a three-phase state as concentration fluctuations grow. 
The initial conditions are set as
	$$
	\begin{aligned}
		& \phi_1(\mathbf{x}, 0)=0.5\left(\frac{y}{2}+0.25\right)+0.001 \operatorname{rand}(x, y), \\
		& \phi_2(\mathbf{x}, 0)=0.5\left(\frac{y}{2}+0.25\right)+0.001 \operatorname{rand}(x, y), \\
		& \phi_3(\mathbf{x}, 0)=1-\phi_1(\mathbf{x}, 0)-\phi_2(\mathbf{x}, 0),
	\end{aligned}
	$$
	where  $\text{rand}(x,y)$ represents random data between $[-1, 1]^2$. 
	The computational domain is $\Omega=[0,2] \times[0,1]$, and we use $256\times 128$ modes to discretize the space variables. The parameters $M=10^{-3}$ and $\epsilon=0.025$ are chosen. 
	Fig.\ref{TCH3} illustrates the evolution of energy and the Lagrange multiplier $eta$ with time for the ternary Cahn–Hilliard phase-field model with $\Delta t=1.0\times 10^{-4}$.
	Additionally, different surface tension strengths, including  $\left(\sigma_{12}, \sigma_{13}, \sigma_{23}\right)=(1,1,1)$ and $\left(\sigma_{12}, \sigma_{13}, \sigma_{23}\right)=(3,1,1)$ are tested and presented in Fig. \ref{TCH1} and Fig. \ref{TCH2}.

\begin{figure}[H]
	\centering
%	\subfigure[]{
		\begin{minipage}[c]{0.65\textwidth}
			\includegraphics[width=1\textwidth]{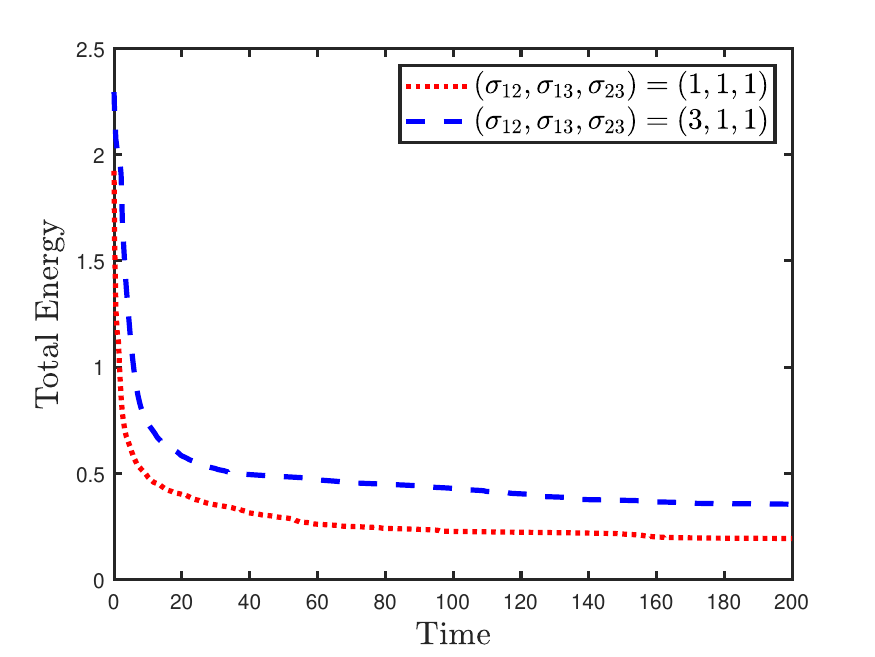}
		\end{minipage}
%	}
%	\subfigure[]{
%		\begin{minipage}[c]{0.4\textwidth}
%			\includegraphics[width=1\textwidth]{TCHEeta.eps}
%		\end{minipage}
%	}
	{\caption{The temporal evolution of energy and the Lagrange multiplier $eta$ for the ternary Cahn–Hilliard phase-field model.}\label{TCH3}}
\end{figure}
	
\begin{figure}[htp]
	\centering
	\subfigure[T=5]{
		\begin{minipage}[c]{0.3\textwidth}
			\includegraphics[width=1\textwidth]{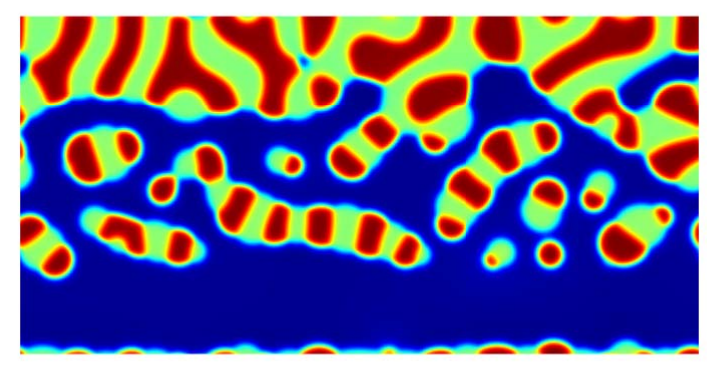}
		\end{minipage}
	}
	\subfigure[T=10]{
		\begin{minipage}[c]{0.3\textwidth}
			\includegraphics[width=1\textwidth]{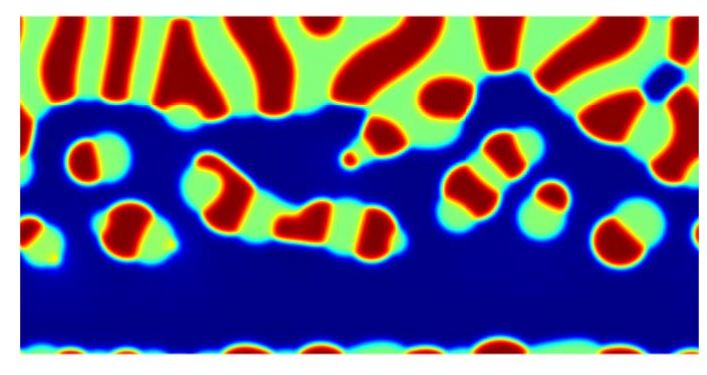}
		\end{minipage}
	}
	\subfigure[T=20]{
		\begin{minipage}[c]{0.3\textwidth}
			\includegraphics[width=1\textwidth]{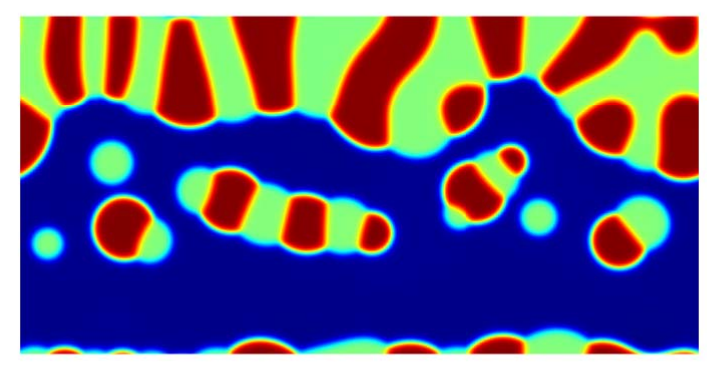}
		\end{minipage}
	}
	\subfigure[T=50]{
		\begin{minipage}[c]{0.3\textwidth}
			\includegraphics[width=1\textwidth]{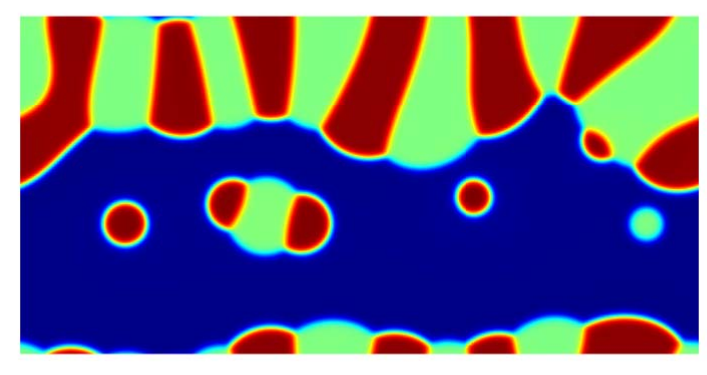}
		\end{minipage}
	}
	\subfigure[T=100]{
		\begin{minipage}[c]{0.3\textwidth}
			\includegraphics[width=1\textwidth]{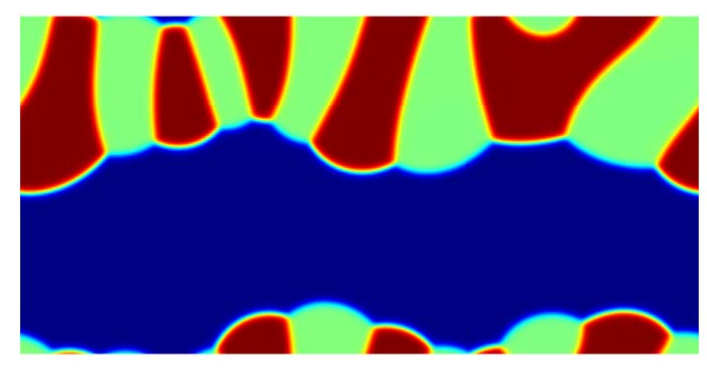}
		\end{minipage}
	}
	\subfigure[T=200]{
		\begin{minipage}[c]{0.3\textwidth}
			\includegraphics[width=1\textwidth]{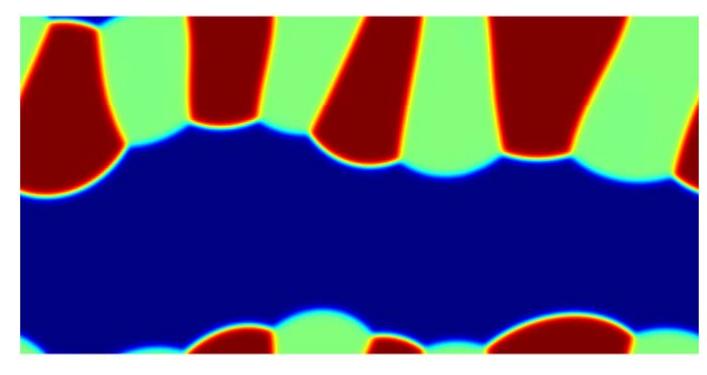}
		\end{minipage}
	}
	{\caption{Dynamical evolution of the profile $\frac{1}{2}\phi_1+\phi_2$ for the spinodal decomposition examples with $\left(\sigma_{12}, \sigma_{13}, \sigma_{23}\right)=(1,1,1)$. }\label{TCH1}}
\end{figure}

\begin{figure}[htp]
	\centering
	\subfigure[T=5]{
		\begin{minipage}[c]{0.3\textwidth}
			\includegraphics[width=1\textwidth]{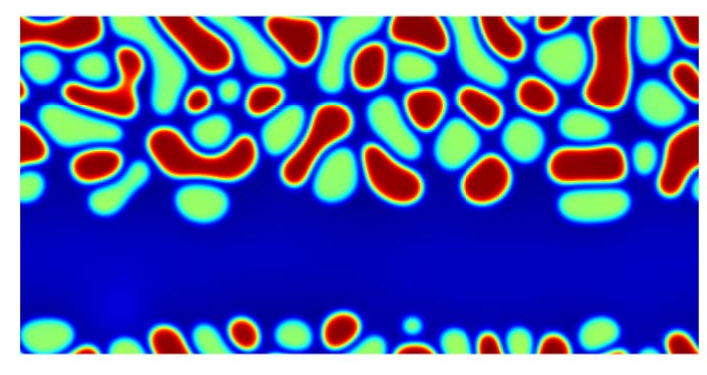}
		\end{minipage}
	}
	\subfigure[T=10]{
		\begin{minipage}[c]{0.3\textwidth}
			\includegraphics[width=1\textwidth]{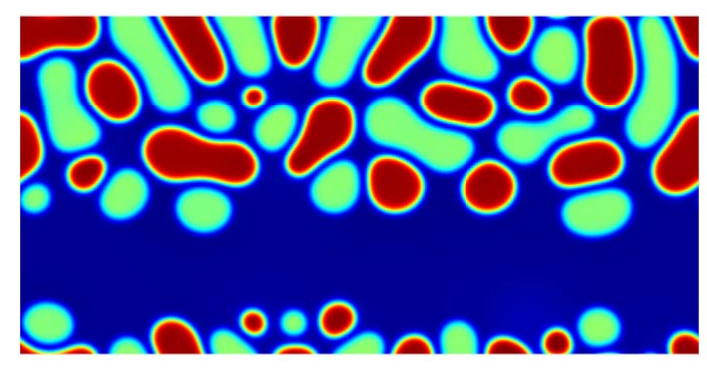}
		\end{minipage}
	}
	\subfigure[T=20]{
		\begin{minipage}[c]{0.3\textwidth}
			\includegraphics[width=1\textwidth]{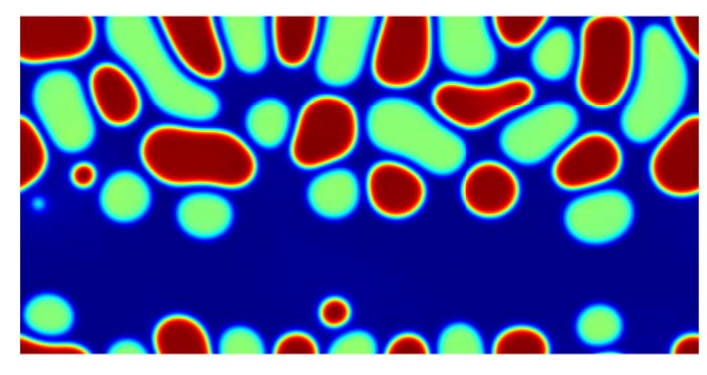}
		\end{minipage}
	}
	\subfigure[T=50]{
		\begin{minipage}[c]{0.3\textwidth}
			\includegraphics[width=1\textwidth]{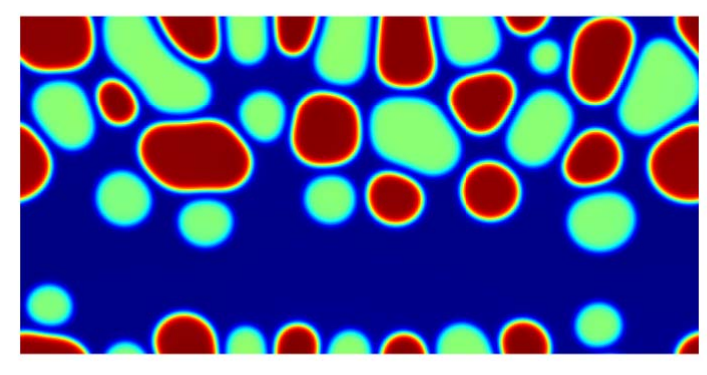}
		\end{minipage}
	}
	\subfigure[T=100]{
		\begin{minipage}[c]{0.3\textwidth}
			\includegraphics[width=1\textwidth]{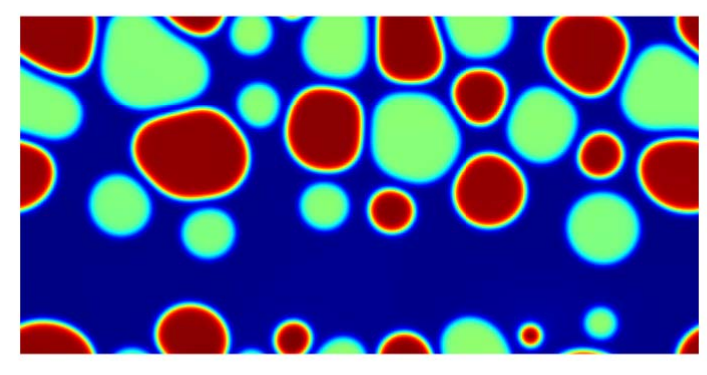}
		\end{minipage}
	}
	\subfigure[T=200]{
		\begin{minipage}[c]{0.3\textwidth}
			\includegraphics[width=1\textwidth]{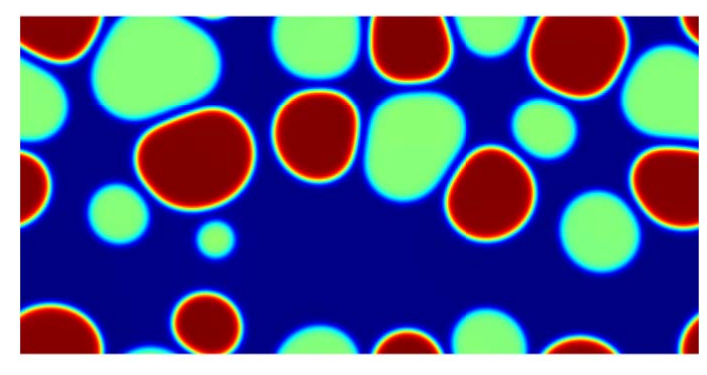}
		\end{minipage}
	}
	{\caption{Dynamical evolution of the profile $\frac{1}{2}\phi_1+\phi_2$ for the spinodal decomposition examples with $\left(\sigma_{12}, \sigma_{13}, \sigma_{23}\right)=(3,1,1)$. }\label{TCH2}}
\end{figure}	
%	The profiles of $\phi_i(i=1,2,3)$ at various times using the SVM-RK4 scheme with a time step $\tau=0.001$ are depicted in Figs. 8 and 9, respectively. For the partial spreading case, some block patterns are observed, while the total scenario forms some bubble-shaped patterns. These observations are consistent with the phase diagram of the ternary Cahn-Hilliard model reported in the literature.
%	
%	In addition, we also present evolution of the original energy functional and supplementary variable $\alpha(t)$ for the two cases in Fig. 10, which confirms that the curves of the original energy decay monotonically with time and the numerical results in $\alpha(t)$ are accurate up to the desired order.
\end{example}

\section{Conclusion}
In this paper, we considered a new and efficient method to modify the recently developed Lagrange multiplier approach in \cite{cheng2020new} for dissipative systems. The new proposed method, called semi-implicit combined Lagrange multiplier approach, includes three steps to simplify the solving process of the introduced Lagrange multiplier to save computational costs. A series of second- and high-order numerical schemes were given sequentially and have been proved to satisfy original dissipative law. In further, we will consider to apply the same algorithm to construct modified SAV approach with unconditionally original energy dissipative law.      
\section*{Acknowledgement}
No potential conflict of interest was reported by the author. We would like to acknowledge the assistance of volunteers in putting together this example manuscript and supplement.
% \section*{References}
\bibliographystyle{siamplain}
\bibliography{Reference}

\end{document}